\documentclass[12pt,leqno]{article}

\usepackage{amsmath,amsfonts,amssymb,amscd,amsthm,indentfirst,graphicx,mathtools,enumerate,bm}
\usepackage{extarrows}

\theoremstyle{plain}
  \newtheorem{prop}{Proposition}[section]
  \newtheorem{thm}[prop]{Theorem}
  \newtheorem{cor}[prop]{Corollary}
  \newtheorem{lem}[prop]{Lemma}
  \newtheorem*{lem*}{Lemma}
  \newtheorem*{claim*}{Claim}
\newtheorem{defn}[prop]{Definition}
  \newtheorem*{defn*}{Definition}
  
  \newtheorem*{rem*}{Remarks}

  \newtheorem*{rema*}{Remark}

  \newtheorem{rems*}{Remarks}

\def\ntn{n\times n}
\def\half{{\tfrac 12}}
\def\Si{\Sigma}
\def\si{\sigma}

\def\p1{P^{-1}}

\def\Th{\Theta}
\def\eq{\eqref}

\def\midl{{\mathop{\mid}\limits}}
\def\La{\Lambda}

\def\be{\begin{equation}}
\def\ee{\end{equation}}
\def\ul{\underline}

\def\ue{\underline{e}}

\def\la{\lambda}

\def\intl{\int\limits}

\def\om{\omega}
\def\Om{\Omega}

\def\de{{\delta}}

\def\rhd{\rightharpoondown}

\def\rhd{\rightharpoondown}

\def\thalf{{\textstyle{\frac 12}}}

\def\k{\kappa}

\def\rank{\text{rank}}

\def\cald{{\cal{D}}}
\def\cale{{\cal{E}}}

\def\calk{{\cal{K}}}

\def\calt{{\cal{T}}}

\def\bbr{\mathbb{R}}
\def\bbs{\mathbb{S}}
\def\cals{\mathcal{S}}

\def\calt{\mathcal{T}}
\def\Ga{\Gamma}

\def\Om{\Omega}
\def\eqadef{{\stackrel{{def}}{=}}}
\def\intl{\int\limits}

\def\suml{\sum\limits}
\def\rhd{\rightharpoondown}

\def\sgn{\mathrm{sgn}}
\def\p{\partial}

\def\pOm{\partial\Omega}

\numberwithin{equation}{section}

\date{} %

\numberwithin{figure}{section}
\newcommand{\vertiii}[1]{{\left\vert\kern-0.25ex\left\vert\kern-0.25ex\left\vert #1
  \right\vert\kern-0.25ex\right\vert\kern-0.25ex\right\vert}}
\begin{document}
\title{On the symmetry group for systems of conservation laws}
\author{
}
\maketitle
\centerline{\large Michael Sever}
\centerline{Department of Mathematics}
\centerline{ The Hebrew University}
\centerline{Jerusalem, ISRAEL}

\bigskip

\begin{abstract}
A case can be made that the utility of quasi-linear systems of conservation laws as physical models is largely limited to Euler system models of fluid flow, at least in higher dimensions.  Qualified corroboration of this conjecture is obtained here, by attempt to extend the class of systems attractive as physical models via the associated symmetry group.
\end{abstract}

\section{Introduction}

First order, quasi-linear systems of conservation laws typically appear as formal limits of vanishing dissipation in physical models containing higher order terms, regarding such models   as singularly perturbed.  Models are thus simplified, and approximation of solutions thereof expedited, at the expense of an enlarged solution class including weak solutions, typically piecewise continuous.  A supplemental admissibility condition is required to recover uniqueness and even  boundedness in the class of weak solutions for given data.

Verification of existence of a vanishing dissipation limit in the class of admissible weak solutions remains limited to special cases, including scalar conservation laws and hyperbolic systems with only two independent variables.  In contrast, abundant empirical evidence, largely obtained using Euler systems and reduced or modified forms thereof, suggests that a vanishing dissipation approximation is more generally possible, extending to higher dimensions and not necessarily requiring hyperbolicity.

Reconciliation with existing theory is arguably obtained \cite{S2} by association of empirical evidence with a weakened form of well-posedness of the corresponding boundary-value problems.  Requisite conditions for such include the existence of an entropy extension of a given system, relation of a given system  to a hyperbolic system and existence of local viscous approximation of discontinuities.  Euler systems and reduced forms thereof emerge as leading if not exclusive candidates for application of this analysis.

One may thus suspect that at least in higher dimensions, applicability of a vanishing dissipation limit is largely limited to Euler systems and reduced or modified forms thereof.  Partial corroboration of this possibility is obtained below.

Within the class of systems of conservation laws, Euler systems are indeed special.  We discuss two classes of Euler systems here, isentropic systems and systems including conservation of energy. Members of each class are determined by an equation of state, and are equipped with an entropy extension  \cite{FL,L}.  Convexity of the entropy density and hyperbolicity of the system follow from simple conditions on the equation of state, independent of the dimension of the system.

In addition to providing an admissibility condition on weak solutions in the form of the familiar entropy inequality, the existence of an entropy extension materially facilitates expression of the nontrivial symmetry group associated with a given system of conservation laws.  Euler systems are distinguished in this context, exclusively admitting rotation and Galilean symmetries \cite{S1}.
On this basis, we speculate that in higher dimensions, applicability of a vanishing dissipation approximation is related to richness of the symmetry group of the system of conservation laws so obtained.

A special class of systems of conservation laws is identified in the following section,  equipped with an entropy extension and  satisfying an additional structural condition.  Expression of the nontrivial symmetry group associated with such systems is further simplified, without restriction on the specific form thereof.  This class of systems is  a natural  form of extension of systems with two independent variables to systems with  higher dimensional domains.

Systems of this class are shown to satisfy attractive properties as regards appearance as vanishing dissipation limits, and as physical models, in the sense of predictable effects of reduction or modification on the associated nontrivial symmetry group.

The hyperbolic members of this class, with maximal richness of the symmetry group,  are shown to correspond to Euler systems.  And in higher dimensions, hyperbolic systems with nonempty but sub-maximal richness of the nontrivial symmetry group correspond to suitable modifications of Euler systems.

\subsection{Notation and review}

Systems of conservation laws of dimension $n \ge 3$ with $m \ge 3$ independent variably are here denoted
\be\label{aa} \Sigma_j (z) = 0, \; \; j = 1, \cdots, n,\ee
\be\label{ab} \Sigma_j(z) \, \eqadef\, \suml^m_{i = 1}  (a_{ij} (z))_{x_i},\ee
\be\label{ac} x \in \Om \subset \bbr^m, \; \; z: \Om \to D \subseteq \bbr^n,
\; \; \, \Om, D \; \hbox{open}. \ee

In \eqref{aa}, \eqref{ab}, \eqref{ac}$, \Sigma $ is  understood as an $n$-row vector distribution within $\Om$; the condition \eqref{aa} holds weakly.  In \eqref{ac}, $z$ may be discontinuous and possibly even unbounded on a set of measure zero.  The $a_{ij} (z)$ are system-specific given functions of the dependent variable $z$.

The symmetry group $\cals$ associated with such a system \eqref{aa} is understood as an essential feature of whatever physical model from which the system derives, and as an ingredient of subsequent analysis thereof \cite{BK}.  Here we identify and discuss a particularly attractive class of systems in this context.  Additional motivation is obtained from \cite{S2}, in which suitability of systems of conservation laws for computational investigation is related to the associated symmetry group.

Elements $\cale \in \cals$ associated with a system \eqref{aa} are denoted here by triples
\begin{align}\label{ad} &\cale \, \eqadef\, (x^*(x), \; z^*(z(x)), N),\nonumber \\
&x^*(x) \in \bbr^m, \; \, z^*(z(x)) \in \bbr^n \; (\hbox{for}\,  z(x) \in D), \; N \in M^{n\times n}\end{align}
such that the infinitesimals $x^* (x), z^*(z)$ imply
\be\label{ae}(\Sigma (z))^* = - \Sigma(z) N,\ee
using
\be\label{aea} (z(x))^* = z^* (z(x)) + z_x(x) x^*(x)\ee
almost everywhere in $\Om$.

Given $x, \Om, z, D$ in \eqref{ac}, any such $\cale$ determines a  one-parameter Lie group of transformations
\be\label{af} \calt_\cale (z, \Om) = \{ \tilde z (\tau; z, x), \tilde x (\tau; x), \tilde \Om (\tau), \tau \in \bbr \}\ee
from
\begin{align}\label{ag} &\tilde x_\tau (\tau; x) = x^* (\tilde x (\tau; x)), \; \, \tilde x(0, x) = x \in \Om\nonumber \\
&\tilde z_\tau (\tau; z, x) = z^* (\tilde z (\tau; z, x)), \; \, x \in \tilde \Om(\tau), \; \;  \tilde z(0; z, x) = z(x) \in D\nonumber \\
&\tilde \Om (\tau) = \{ \tilde x (\tau; x)\}\end{align}
understanding $\tau$ restricted as necessary so that $\tilde z(\tau; z, x) \in D$.

For $z$ satisfying \eqref{aa}, however weakly, it follows from \eqref{ae} that $\calt_\cale$ determines an equivalence class of solutions and domains.

Any system \eqref{ab} is equipped with two trivial classes of symmetries, with $z^*(z)$ vanishing identically.

Translation symmetries with respect to $\mu \cdot x$ correspond to
\be\label{ai} \cale_T = (x^*(x) = \mu \in \cals^m,\; \, z^*(z) = 0,\; N = 0).\ee

Uniform scaling of the independent variables $x$ corresponds to
\be\label{aj} \cale_S = (x^*(x) = x, \, z^*(z) = 0, \; N = -I_n),\ee
using $I$ as identity matrices of indicated dimension throughout.

Of interest here are possible additional, system-specific symmetries denoted  $\cals_R$, such as rotation symmetries for systems with multiple ``space variables" $x_i$, for which $z^*(z)$ does not vanish identically.  In particular, the algebraic structure of the resulting symmetry group $\cals$, with subsets $\cals_T, \cals_S, \cals_R$, is materially complicated by the existence of nonempty $\cals_R$.

Given $\cale_1, \cale_2 \in \cals$, by convention
\be\label{al} \cale_1 + \cale_2 \, \eqadef\, (x^*(x) = x^*_1 (x) + x^*_2 (x), \; z^*(x) = z^*_1 (z) + z^*_2(z), \; N = N_1 + N_2)\ee
satisfies
\be\label{ak} \cale_1 + \cale_2 \in \cals\ee
only if the corresponding $\calt_1, \calt_2$ commute,
\be\label{am} \calt_1 \circ \calt_2 = \calt_2 \circ \calt_1.\ee

With the dependent variable $z$ to be judiciously chosen below, removing additive constants from $x, z$ as appropriate, attention is here  restricted to  $\cale \in \cals_R$ of the form
\be\label{ba} \cale = ( x^*(x) = Xx, z^* (z) = (Zz+\om),  N), \; X \in M^{m\times m}, Z\in M^{n\times n}, \, \om \in\bbr^n. \ee

Such $\cale$ are applied to \eqref{ae} using
\begin{align}\label{bb} &f^* (x) = f_x (x) x^*(x) = f_x (x) Xx\nonumber \\
&g^*(z) = g_z (z) z^*(z) = g_z(z) (Zz+ \om);\\ 
\label{bc} &(f_x(x))^* = (f^*(x))_x - f_x(x)X\nonumber \\
&(g_z(z))^* =( g^*(z))_z - g_z (z) Z\nonumber \\
& (g_{zz} (z))^* = (g^*(z))_{zz} - Z^\dag g_{zz} (z) - g_{zz} (z)Z.\end{align}

By convention, for components of vector functions
\be\label{bd} (f_i(\cdot))^* \,\eqadef\, ((f(\cdot))^*)_i;\ee
 for $x, z$-derivatives
\be\label{be} (f_{x_i} (x))^*\,\eqadef\, ((f_x(x))^*)_i,  \ee
\be\label{bf} (g_{z_j} (z))^* \,\eqadef\, ((g_z(z))^*)_j;\ee
and for the matrix function $a_{\cdot\cdot}$
\be\label{bfa} a^*_{ij} \, \eqadef\, (a^*)_{ij}.\ee

A sufficient condition for \eqref{ae} is the following.

\begin{lem} Assume $X\in M^{m\times m}, \; Z\in M^{n\times n}, \, \om \in \bbr^n,  c_Z \in \bbr $ such that using \eqref{bfa} in the form
\be\label{bfb} a^*_{ij} = a_{ij, z} z^* = a_{ij, z} (Zz+ \om),\ee
it follows that
\be\label{bg} a(z)^* = (X + c_Z I_m) a(z) - a(z) Z.\ee

Then $\cale \in \cals_R, \; \cale$ of the form \eqref{ba} with
\be\label{bh} N = Z - c_Z I_n.\ee
\end{lem}
\begin{rema*} The set of $(X, Z, \om)$ satisfying \eq{ba}, \eq{bg} is a linear vector space.
\end{rema*}
\begin{proof}  We introduce test functions for the weak form of \eqref{aa},
\be\label{bi} \theta : \Om \to \bbr^n\ee
sufficiently smooth, vanishing on the boundary $\pOm$, otherwise arbitrary.  By convention, throughout, for any $\cale \in \cals_R$,
\be\label{bl} \theta^* = Z\theta.\ee

Then the weak form of \eqref{aa}, \eqref{ab}, after partial integration
\be\label{bj} \intl_\Om \Si(z)\theta = - \intl_\Om \suml^m_{i = 1} \suml^n_{j = 1} a_{ij} (z) \theta_{j, x_i},\ee
satisfies, for any $\cale \in \cals_R$,
\be\label{bk} \intl_\Om (\Si (z)\theta)^* = \intl_\Om ((\Si(z))^*\theta +\Si(z) \theta^*).\ee

Using \eqref{bj} in \eqref{bk}, then \eqref{bd}, \eqref{be}, \eqref{bfa}, then \eqref{bg}, \eqref{bl}, \eqref{bc},

\begin{align}\label{bm}
 (\suml^m_{i = 1} &\suml^n_{j = 1} a_{ij} \theta_{j, x_i})^* = \suml^m_{i = 1} \suml^n_{j = 1} (a^*_{ij} \theta_{j, x_i} + a_{ij} (\theta_{j, x_i})^*)\nonumber \\
&= \suml^m_{i = 1} \suml^n_{j = 1} ((a^*)_{ij} \theta_{j, x_i} + a_{ij} (\theta_x)^*_{ji}\nonumber \\
&= \suml^m_{i = 1} \suml^n_{j = 1}(((X + c_Z I_m) a-aZ)_{ij} \theta_{j, x_i}\nonumber \\
&\qquad \quad + a_{ij} ((Z\theta)_{j, x_i} - (\theta_{j, x} X)_i))\nonumber \\
&= \suml^m_{i = 1} \suml^n_{j = 1}(( \suml^m_{k = 1} X_{ik} a_{kj} + c_Z a_{ij} - \suml^n_{l = 1} a_{i_l} Z_{lj}) \theta_{j, x_i}\nonumber \\
&\qquad \; \; + a_{ij} (\suml^n_{l = 1} Z_{jl} \theta_{l, x_i} - \suml^m_{k = 1} \theta_{j, x_k} X_{ki}))\nonumber \\
&= c_Z \suml^m_{i = 1}  \suml^n_{j = 1} a_{ij} \theta_{j, x_i},\end{align}
judiciously interchanging $i$ and $k$, \ $j $ and $l$ in the last step.

Now $\cale \in \cals_R$ and \eqref{bh} follow from \eqref{bm}, \eqref{bk}, \eqref{bl}.
\end{proof}

The condition \eqref{bg} may be used directly to verify (and in principle determine) the symmetry group $\cals_R$.  The procedure is materially simplified, however, for systems \eqref{ab} equipped with an entropy extension [L], not necessarily convex.

An entropy extension is an $m$-vector function $q$ with components
\be\label{bn} q_i = q_{i} (a_{i1}, \ldots, a_{in}), \; i = 1, \cdots, m,\ee
of class $C^1$
such that sufficiently smooth solutions \eqref{aa} satisfy an additional conservation law
\be\label{bo} \suml^m_{i = 1} (q_i (a_{i \cdot} (z)))_{x_i} = 0.\ee

Such mandates the choice of  $z$, in \eqref{ab}, \eqref{ac} and throughout, such that for sufficiently smooth $z$, formally,
\be\label{bp} \suml^m_{i = 1} (q_i (a_{i\cdot} (z)))_{x_i} = \suml^n_{j = 1} \Sigma_j (z) z_j\cdot, \ee
thus identifying $z$ [G,M] as the ``symmetric" dependent variable for the given system,  satisfying
\be\label{br} \frac{\partial q_i(a_{i\cdot })}{\partial a_{ij}} = z_j, \; \, i = 1, \cdots, m, \; \, j = 1, \cdots, n\ee
excluding $i, j$ such that $a_{ij}$ vanishes identically.

Then the $m$-vector Lagrange dual
\begin{align}\label{bs} \psi_i(z) & \eqadef  \suml^n_{j = 1} a_{ij} \frac{\partial q_i}{\partial a_{ij} } - q_i (a_{i\cdot})\nonumber \\
& = \suml^n_{j = 1} a_{ij} (z) z_j - q_i (a_{i\cdot} (z))\end{align}
satisfies
\be\label{bt} \frac{\partial\psi_i(z)}{\partial z_j} = a_{ij} (z), \; \; i = 1, \cdots, m, \; \, j = 1, \cdots, n,\ee
having tacitly assumed
\be\label{bta} \psi\in C^2(D \to \bbr^m).\ee

The system $\Sigma(z)$ is hyperbolic, with
\be\label{btb} e_H \in \bbr^m\ee
denoting a time-like direction, if
\be\label{btc}  (e_H \cdot \psi)_{zz} (z) \ge 0\ee
in the sense of $n$-matrices, with strict inequality almost everywhere in $D$.   Such is equivalent to $e_H \cdot q$ convex in the components of $e^\dag_H a$.

For systems \eqref{aa}, \eqref{ab} equipped with an entropy extension, expression of the nontrivial symmetry group \eqref{bg} is materially simplified using the potential function $\psi(z)$.  A convexity condition on $q$ or $\psi$ is not required in this context.

By convention, trivial symmetries $x^*$ vanishing, $z^*$ constant, using \eq{bt}, are not included in the nontrivial symmetry group $\cals_R$.

The following is extension of a previous  result \cite{S1}.

\begin{thm} Given \eqref{bt}, a sufficient condition for \eqref{bg} is
\be\label{bu} \psi(z)^* = (X + c_Z I_m) \psi(z).\ee
\end{thm}
\begin{rema*} The arbitrary uniform scaling of $z, q, \psi, \om$ in \eqref{bn}, \eqref{bo}, \eqref{bp}, \eqref{br}, \eqref{bs}, \eqref{bt}, \eqref{ba} is unimportant and ignored throughout.
\end{rema*}
\begin{proof} This follows immediately using \eqref{bt}, \eqref{bc}, \eqref{bfa}.
\end{proof}

An arbitrary additive linear multiple of $z$ in the expression for $\psi(z)$ satisfying \eqref{bt}, \eqref{bs} is removed by the condition \eqref{bu}.

Using \eqref{bc}, \eqref{be}, \eq{ab}, \eq{bt}  in \eqref{bs}
\be\label{bv} q^* = (X + c_Z I_m)q + \psi_z\om.\ee

In \eqref{bv} using \eqref{bc}, \eqref{be}
\be\label{bw} (\triangledown \cdot q)^* = c_Z \triangledown \cdot q + \Si(z)\om,\ee
so the familiar entropy inequality
\be\label{bx} \triangledown \cdot q \le 0\ee
is invariant under $\cals_R$ in this generality, using \eq{aa}.

\subsection{Related systems}

A given system \eqref{aa}, \eqref{ab} determines an equivalence class of systems
\be\label{ap}\Sigma^{QU}_j (z^Q)_j \, \eqadef\,\suml^m_{i = 1} (a^{QU}_{ij} (z^Q))_{x^U_i}, j = 1, \dots, n,\ee
from linear combinations of the $x_i, a_{ij}$, of the form
\be\label{an} x^U = Ux,\ee
\be\label{ao} a^{QU} (z^Q) = U a(z) Q^{-1},\ee
for arbitrary nonsingular $m$-matrix $U$, $n$-matrix $Q$.

Comparing \eqref{ap} with \eqref{ab}, using \eqref{an}, \eqref{ao},
\be\label{aq} \Si^{QU} (z^Q) = \Si(z) Q^{-1},\ee
making solution sets of \eqref{ab}, \eqref{ap} isomorphic.

For systems $\Si$ such that \eqref{bt} holds, the condition \eqref{ao} is satisfied with
\be\label{ana} a^{QU}_{i\cdot} (z^Q) = \psi^U_{i, z^Q} (z^Q) = \psi^U_{i, z}(z^Q)Q^{-1} ,
\ee
\be\label{anb} z^Q = Qz,
\ee
\be\label{anc} \psi^U (z^Q) = U \psi(z).
\ee

From \eqref{bs}, \eqref{ana}, \eqref{anb}, \eqref{anc}, the corresponding entropy extension is
\be\label{and} q^U (a^{QU}_{\cdot\cdot}) = Uq(a_{\cdot\cdot});
\ee
the entropy inequality \eqref{bx} is unaffected.

Whatever nontrivial symmetry group $\cals_R$ survives such transformations, making similarity transformations
\be\label{ar} X^U = UXU^{-1}, \; \; Z^Q= QZQ^{-1}, \om^Q = Q\om\ee
in \eqref{ba}.  The constant $c_Z$ in \eqref{bg}, \eqref{bh}, \eqref{bu}, \eqref{bw} is also unaffected.

Also determined from a ``primitive" system $\Si $ \eqref{ab} are various related systems $\Si'$, which may or may not be of the same form, generally with different dimensions.  The solution set of such $\Si'$ will be related, but not isomorphic, to that of $\Si$, thus relating the associated symmetry group $\cals'$ for $\Si'$ to $\cals$ for $\Si$.

Familiar examples of $\Sigma'$ are reduced systems, satisfied by the subset of solutions of \eqref{aa}, \eqref{ab} invariant under some $\cale_1 \in \cals$, satisfying
\be\label{da} (z(x))^* = z^* (z(x)) + z_x (x) x^* (x) = 0\ee
using \eqref{aea} for $\cale_1$ in \eqref{ad}, equivalently
\be\label{db} \tilde z (\tau; z, x) = z(x)\ee
in \eqref{ag}.  Elements $\cale_2 \in \cals'_R$  must then be such that \eqref{am} holds.

For $\cale_1 = \cale_T$ in \eqref{ai}, $\Sigma'$ determining solutions of \eqref{aa}, \eqref{ab} stationary with respect to $\mu \cdot x$, satisfying
\be\label{dc} z_x(x) \mu = 0\ee
 almost everywhere in $\Om$,
the condition \eqref{am} requires $\cale_2$ such that $x^*(\cdot ) $ is independent of $\mu\cdot x$,
equivalently
\be\label{de} X_2 \mu = 0\ee
using \eqref{ba}.

For $\cale_1 \in \cals_R$, using \eqref{ba}, the condition \eqref{am} becomes
\be\label{dg} X_1 X_2 = X_2 X_1, \; \; Z_1 Z_2 = Z_2 Z_1,\; \;  Z_2 \om_1 + \om_2
 = Z_1 \om_2 + \om_1.\ee

Such applies also to $\cale_1 = \cale_S$ in \eqref{aj}, identifying $X_1 = I_m, Z_1 = 0, \om_1 =\om_2 = 0$. Then \eqref{dg} holds trivially.
Self-similar solutions $z \big(\frac{x}{e \cdot x}\big), \; \,  
e \in \bbr^m$ in \eq{aa}, \eq{ab}, satisfy \eq{da} with $(e\cdot x)^* = 0$ and $x^*$ independent of $e\cdot x$.  Such solutions thus inherit elements of $\cals_R$ with
\be\label{dga} Xe = X^\dag e = 0.\ee

The primitive system may also be simplified by fixing some component of $z$, setting
\be\label{dh} e \cdot z = c_e, \; \; c_e \in \hat I\ee
for some $e \in \bbs^n$ and some open interval $\hat I \subseteq \bbr$,  dropping the component $\Sigma (z) e$ from \eqref{aa}, retaining a system of dimension $n-1$. The nontrivial symmetry group for the resulting system must then be such that
\be\label{di} (e\cdot z)^* = 0.\ee

If
\be\label{dha} c_e \neq 0\ee
in \eqref{dh}, then using \eqref{ba}, the condition \eqref{di} requires, for any $\cale \in \cals'_{R}$,
\be\label{dj} Z^\dag e = 0, \ee
\be\label{djb} e \cdot \om = 0.\ee

The primitive system may also  be modified by conserving entropy,
exchanging the entropy equality \eqref{bo} with one of the constituent equations, say $\Si_l (z)$, for some $l = 1, \dots$, minimum $(m, n)$.

For $c_e$ satisfying \eq{dha},  $D $ restricted  by
\be\label{dka} z_l \neq 0, \; \;  \sgn (z_l) = \sgn (c_e), \; \; \, z \in D,\ee
the image
\be\label{dja} \tilde \Si (\tilde z ) \, \eqadef\, \calt_{q, l, c_e} \Si(z)\ee
satisfies
\begin{align}\label{djz} \tilde a_{ij} (\tilde z) & = a_{ij} (z), \; \; \, i = 1, \cdots, m, \; \; j = 1, \cdots, n; \; \; i, j \neq l;\nonumber \\
\tilde a_{i_l} (\tilde z) & = - c_e q_i (a_{i\cdot} (z)), \; \; \, i = 1, \cdots, m, \; \; i \neq l\nonumber \\
\tilde q_i (\tilde a_{i\cdot} (\tilde z)) & = - \frac{1}{c_e} a_{i_l}, \; \; \, i = 1, \cdots, m, \; \; \,  i \neq l.\end{align}

Such corresponds to exchange of dependent variable
\be\label{dkb} \tilde z_j =  c_e \,\frac{z_j}{z_l}, \; \; j = 1,\dots, n, \; \; j \neq l\ee
\be\label{dkc} \tilde z_l =
 \frac{c^2_e}{z_l},\ee
and potential functions in \eqref{bf} for $\tilde \Si$,
\begin{align}\label{dkd} &\tilde \psi_i (\tilde z) = \frac{c_e}{z_l} \, \psi_i (z), \; \; \, i = 1,\cdots, m, \; \; i \neq l;\nonumber \\
&\tilde \psi_l (\tilde z) = - \frac{1}{z_l} \, \suml^n_{j = 1} z_j \psi_j (z).\end{align}

The remaining density/flux functions for $\tilde \Si (\tilde z)$ are obtained from \eq{dkd}, using \eq{bt},
\begin{align}\label{djy}
\tilde a_{lj} (\tilde z) & = - \frac{1}{z_l} \big( \suml^n_{k = 1} z_k a_{kj} (z) + \psi_j (z)\big), \; \; j = 1, \cdots, n, \; \; j \neq l;\nonumber \\
\tilde a_{ll} (\tilde z) &=  - \frac{1}{z_l} \suml^n_{k = 1} z_k a_{kl} (z) + \frac{1}{z^2_l}  \suml^n_{{k = 1, k \neq l}}  z_k \psi_k(z).\end{align}

A necessary condition for \eqref{dka} is $z^*_l $ vanishing identically.  Such requires each $(Z,\om)$ in \eqref{ba} satisfying
\be\label{dke} Z_{l j} = 0, \; \; j = 1, \dots, n,\; \om_l = 0.\ee

Sufficient conditions for \eq{bu} remain to be determined.

Alternative modifications of a primitive system, including extension to higher dimensions, coupling of multiple systems, and adding dissipation, require more extensive  discussion  and are not treated in this generality.

\section{A special class of systems}

Using \eqref{bb} with $g = \psi$ and \eqref{bt}, the condition  \eqref{bu} becomes
\be\label{bua} a(z)(Zz + \om) = (X + c_Z I_m) \psi (z),\ee
linearly relating $X,Z,\om,  c_Z$ for given $z, a(z), \psi(z)$. The relationship is dramatically simplified for a special class of systems \eqref{ab},
satisfying an additional structural condition.

For these systems
\be\label{ca} m = n,  \ee
and  the following apply.

\begin{thm} For a system satisfying \eqref{ca}, \eq{bta},  existence of a nonvanishing scalar function
$\xi\in C^2 (D \to \bbr_+)$ such that $\hat  a(z)$ is a symmetric $n$-matrix function of $z$, with
\be\label{cb} \hat  a_{ij} (z) \, \eqadef \, a_{ij} (z) - \frac{\psi_i (z) \xi_{z_j} (z)}{\xi(z)}\ee
is equivalent to the existence of scalar functions
\be\label{ccz} \xi \in C^2 (D \to \bbr), \; \; \zeta \in C^3 (D \to \bbr),\ee
$\xi$ nonvanishing, such that
\be\label{cc} \psi(z) = \xi(z) \zeta_z(z)^\dag.\ee
\end{thm}

\begin{rem*}  Such $\xi, \zeta$ are not unique, as \eqref{cc} holds for any $\underline{\xi}, \underline{\zeta}$
\be\label{cd} \underline{\zeta}(z) = f(\zeta(z)), \; \, \underline{\xi}(z) = \xi(z) / f_\zeta  (\zeta(z)),\ee
with nonvanishing $f_\zeta (\cdot)$.  Additionally, there exists an arbitrary additive constant in $\zeta(z)$.

Nonvanishing $\xi(z)$ implicitly restricts $D$ for a given system.

\end{rem*}
\begin{proof} Assuming \eqref{cb}, $\zeta(z)$ is determined from the symmetric $n$-matrix function
\begin{align}\label{ce} \frac{\hat  a (z)}{\xi(z)} &= \Big( \frac{\psi(z)}{\xi(z)} \Big)_z\nonumber \\
&\eqadef \zeta_{zz} (z),\end{align}
using \eqref{bt}, determining $\zeta (z)$ up to an affine function of $z$.  The condition  \eqref{cc} follows from the choice
\be\label{cf} \zeta^\dag_z (0) = \frac{\psi(0)}{\xi(0)};\ee
an additive constant in $\zeta$ is immaterial in \eqref{cc}.

Assuming \eqref{cc}, differentiation and use of \eqref{bt}, \eqref{cc} gives
\begin{align}\label{cg}  a(z) &=(\xi(z) \zeta_z (z)^\dag)_z\nonumber\\
&= \xi (z) \zeta_{zz} (z) + \zeta_z (z)^\dag \xi_z (z)\nonumber \\
 &=\xi(z)\zeta_{zz}(z) + \frac{\psi(z)\xi_z(z)}{\xi(z)},\end{align}
from which \eqref{cb} follows.\end{proof}


 Regularity \eqref{ccz} is assumed throughout, thus satisfying \eqref{bta} using \eqref{cc}.

The condition \eq{cc} is  one form of extension of a given system \eq{aa}, \eq{ab}, \eq{ac} with
\be\label{cia} m = 2, n > 2\ee
to a system satisfying \eq{ca}. A restriction of \eq{cc},
\be\label{cib} \psi_i (z) = \xi(z) \zeta_{z_i} (z), \; \; \, i = 1, 2, \ee
is obtained from solution of
\be\label{cic} \Big( \frac{\psi_1(z)}{\xi(z)}\Big)_{z_2} \; = \; \Big(\frac{\psi_2(z)}{\xi(z)}\Big)_{z_1}\ee
pointwise with respect to $z_3, \cdots, z_n$, restricting $(z_1, z_2) \in \bbr^2$ as required.  Then for suitable $D \subseteq \bbr^n, \; \; \psi: D \to \bbr^n$ is obtained satisfying \eq{cc}.  The original system satisfying \eq{cia} is thus the reduced form of the system so extended, stationary with respect to $x_3, \cdots, x_n$.

The extension is not unique, as follows from lack of uniqueness in solution of \eq{cic}.

 \begin{thm} Assume a system satisfying \eq{ca}, \eq{cc}, $Z \in M^{n\times n}, \om \in \bbr^n$  such that using \eqref{ba}, \eq{bb}, for some constants $c_Z, c_{\xi,Z}, c_{\zeta,Z}$, independent of $z$,
 \begin{align}\label{cj} \zeta^* (z) &= \zeta_z (z) (Zz+ \om)\nonumber \\
 &=(c_Z - c_{\xi, Z}) \zeta(z) + c_{\zeta, Z}.\end{align}

 Then a necessary and sufficient condition  for \eq{bu} with
 \be\label{ch} X = - Z^\dag\ee
 is that, independently of $z$,
 \begin{align}\label{ci} \xi^*(z) &= \xi_z(z)(Zz+\om) \nonumber \\
 &=c_{\xi, Z} \xi(z).\end{align}

\end{thm}
\begin{proof} From \eq{cc}, using \eq{bc}, \eq{cj}
\begin{align}\label{cl} \psi^*(z) &= \xi^*(z) \zeta^\dag_z(z) + \xi(z) \, (\zeta^\dag_z(z))^*\nonumber \\
&= \xi^* (z) \zeta^\dag_z(z) + \xi(z) ((\zeta^*(z)_z)^\dag - Z^\dag \zeta_z(z)^\dag )\nonumber \\
&= \Big( \big( \frac{\xi^*(z)}{\xi(z)} - c_{\xi, Z} \big) I_n - Z^\dag\Big) \psi(z).\end{align}

The conclusion follows by comparison of \eq{cl} with \eq{bu}.\end{proof}

Except where stated otherwise, for systems $\Si(z)$ satisfying \eq{cc},
attention is restricted to systems for which $\zeta(z)$ in \eq{cc} is such that there is no  nonzero $ e_\perp \in \bbr^n$
\be\label{cga} \zeta_z(y)  e_\perp = 0 \ee
for all $y \in D$.


Both conditions \eqref{ca}, \eqref{cb} admit qualified extension.

\begin{cor} Assume \eqref{ca} and the potential function $\psi(z)$ satisfying \eqref{bt} of the form
\be\label{ct} \psi(z) = \suml^\calk_{\k = 1} \xi_\k (z) \zeta_{\k, z} (z)^\dag\ee
for some positive integer $\calk \ge 2$.

Then \eqref{bh}, \eqref{ch} hold for any $Z,\om$ such that \eqref{ci}, \eqref{cj} hold for each $\xi_{\k}, \zeta_{\k}$ with $c_Z   $ independent of $\k$.
\end{cor}

Nonvanishing $\xi_{\k} $ is not required here.

In view of \eqref{ch}, \eqref{bu}, a necessary condition that a given system \eqref{ab} can be extended to satisfy corollary 2.3 is that for all $X, Z, \om $ in \eqref{ba} for the given system, there exists a constant $c_{XZ}$ such that
\be\label{cta} X_{ij} = - Z_{ji} + c_{XZ}, \; \; i, j = 1, \cdots, \hbox{\, minimum \ } (\hat m, \hat n)\ee
with $\hat m, \hat n$ the values of $m, n$ in \eqref{aa}, \eqref{ab}, \eqref{ac} for the given system, which need not be equal.

For such systems, satisfying one of the conditions \eq{cc}, \eq{ct},  with \eqref{ch}, \eqref{ci}, \eqref{cj} determining $X, c_Z $ from $Z, \om $, the nontrivial symmetry group $\cals_R$ is isomorphic to the set $\{(Z, \om) \}$.  Below we shall use these terms interchangeably and refer to such systems as $Z$-systems.

\begin{thm} For a $Z$-system, $\cals_R$ and $\{(Z, \om) \}$ are linear vector spaces.
\end{thm}
\begin{rem*} For $\cale_1, \cale_2 \in \cals_R$, the condition \eqref{am} or \eqref{dg}, is not required.
\end{rem*}
\begin{proof} For the case \eq{cc}, given  $(Z_1, \om_1), (Z_2, \om_2)  \in \{ (Z,\om)\}$, from \eqref{ci}
\begin{align}\label{cr} \xi^* \, &\, = \, \xi_z (( Z_1 + Z_2) z+\om_1 +\om_2)\nonumber \\
&=(c_{\xi, Z_1} + c_{\xi, Z_2})\xi,\end{align}
and from \eqref{cj}
\begin{align}\label{cs} \zeta^* \, & =\,
\zeta_z ((Z_1 + Z_2)z  +\om_1 +\om_2)\nonumber \\
&=(c_{Z_1} - c_{\xi, Z_1}) \zeta + c_{\zeta, Z_1} + (c_{Z_2} - c_{\xi, Z_2}) \zeta + c_{\zeta, Z_2}.\end{align}

From \eqref{cr}, \eqref{cs}, the conditions \eqref{ci}, \eqref{cj} hold for $Z_1 + Z_2$,
 \be\label{cm} ( Z_1 + Z_2, \om_1 + \om_2)  \in \{( Z,\om) \}\ee
 with
 \be\label{cn} c_{Z_1 + Z_2} = c_{Z_1} + c_{Z_2}\ee
 \be\label{co} c_{\zeta, Z_1 +Z_2} = c_{\zeta, Z_1} + c_{\zeta, Z_2}.\ee

 An entirely analogous argument applies to systems satisfying \eq{ct}.
  \end{proof}

 From \eqref{cn}, the condition \eqref{bh} holds with
 \be\label{cp} N_{Z_1 + Z_2} = N_{Z_1} + N_{Z_2}.\ee

We introduce subspaces
\be\label{cx}\{ (Z, \om)\} = \{(Z, \om)\}_0 \oplus \{ ( Z, \om)\}_\zeta \oplus \{ (Z, \om)\}_c\ee
  such that for $(Z,\om) \in \{ (Z, \om)\}_0$,
 \be\label{cu} \xi^* =  0,\; \; \zeta^* = 0, \; \; c_Z = c_{\xi, Z}  = c_{\zeta, Z} = 0\ee
 in \eqref{cj}, \eqref{ci}.
By convention
\be\label{mca} (I_n, 0) \notin \{ ( Z, \om)\}_0.\ee

A quotient space
\be\label{cxa} \{ (Z_\zeta, \om_\zeta)\}_\zeta \; \eqadef \; \{ (Z, \om) \, | \, c_Z = c_{\xi, Z} = 0 \} / \{ (Z,  \om)\}_0 \ee
is of dimension at most one, determined from an element $(Z_\zeta, \om_\zeta)$ for which
\be\label{cxb} \zeta_z (y) (Z_\zeta y + \om_\zeta) > 0\ee
for all $y \in D$.

Any quotient space
\be\label{cxc} \{ (Z, \om)\}_c \; \eqadef \; \{ ( Z, \om)\} / (\{ (Z, \om)\}_0 \oplus \{ (Z,\om)\}_\zeta\ee
is of dimension not exceeding two, determined from $(Z,\om)$ for which
\be\label{cxd} c_Z = c_{\xi, Z} = 1\ee
or
\be\label{cxe} c_Z = 1, c_{\xi, Z} = 0. \ee

  In examples below,
   $\{ Z, \om \}_c$ will be associated with scaling of $z$, with $Z$ diagonal, $\om$ vanishing. The set  $\{( Z, \om)\}_c$  may be empty depending on the nonlinearity of the given system, and will be largely ignored.

\subsection{Transformations on $Z$-systems}

Transformations on $Z$-systems determine related $Z$-systems, as described in section 1.2.

A (real-valued) transformation
\be\label{cjz} \calt_{QU} \Si(z)\, \eqadef \, \Si^{QU} (z^Q)\ee
as obtained from \eqref{an}, \eq{anb}, determines a $Z$-system with
\be\label{cja} \xi^Q(z^Q) = \xi(z), \; \; \, \zeta^Q(z^Q) = \zeta(z)\ee
if and only if $Q,U$ are related by
\be\label{cjb} U = Q^{-\dag}.\ee

The corresponding elements of $\{ ( Z,\om)\}_0$ are obtained from \eq{ar}; the values of $c_Z, c_{\xi, Z},\;  c_{\zeta, Z}$ in \eq{cj}, \eq{ci} are unchanged.

A transformation $\calt_{ e, c_e}  $ results from application of \eq{dc}, \eq{dh} simultaneously to a $Z$-system $\Si(z)$, of dimension $n$, with
\be\label{cjd} \mu = e, \; \; c_e \in \bbr.\ee

In particular,  the components $e'\cdot z, \; \; e' \perp e$, are unaffected.

The image $\calt_{e, c_e} \Si(z)$, of dimension $n-1$, is thus the reduced form of $\Si(z)$ stationary with respect to $e\cdot x$, and simplified by application of \eqref{dh} in the expressions for $\zeta(z), \xi(z)$, including whatever expression for $z_\zeta (z)$ in \eq{me} below.

Employing a transforation $\calt_{QU}$, without loss of generality by convention below
\be\label{cje} e = \hat e_n\ee
in \eq{djb}, \eq{dj}, \eq{cjd}.  Here and throughout we denote unit vectors in phase space or in $\Omega$ by $\hat e_j$, $ j = 1, \cdots$.
With the convention \eq{cje}, the image
\be\label{uaa} \Si' (z') \, \eqadef \, \calt_{\hat e_{n}, c_e} \Si(z), \; \; \;  z' \, \eqadef \, \begin{pmatrix} z_1\\ \vdots\\ z_{n-1}\end{pmatrix} \ee
is determined with
\be\label{uab} \zeta' (z') = \zeta{z'\choose c_e},  \; \; \, \xi' (z') = \xi {z'\choose c_e}.\ee

Then for all $Q, U$ satisfying \eq{cjb} and
\be\label{uac} Q_{nj} = U_{jn} = \delta_{j,n} , \; \; \; j = 1, \cdots, n\ee
employing the Kroneker delta, there exists $Q', \;  U'$ of dimension $n-1$ satisfying
\begin{align}\label{uad} \calt_{\hat e_n, c_e} \calt_{QU}\; \Si (z) &= \calt_{Q'U'} \, \calt_{\hat e_n, c_e} \Si(z) \nonumber \\
&= \calt_{Q'U'} \, \Si'(z');\end{align}
\be\label{uae} Q'_{ij} = Q_{ij}, \, U'_{ij} = U_{ij}, \; \; \, i, j = 1, \cdots, n-1,\ee
independent of $c_e$.

  Then the elements of $\{ ( \hat Z, \hat \om)\}_0$ for the image $\calt_{e, c_e}\Si(z)$ are obtained from those of $\{ (Z, \om)\}_0$ for $\Si(z)$ satisfying
\be\label{cjf} Z_{nj} = 0, \; \; \, j = 1, \dots, n, \; \; \om_n = 0,\ee
by
\be\label{cjg} \hat Z_{ij} = Z_{ij}, \; \; \, i, j = 1, \dots, n-1\ee
\be\label{cjh} \hat \om_i = \om_i + Z_{in} c_e, \; \; \, i = 1,\cdots, n-1,\ee
using the constant $c_e$ in \eqref{dh}.

For transformations \eq{dja} applied to $Z$-systems $\Si(z)$, the conditions \eq{dkd} are obtained with
\be\label{cjc} \tilde \xi (\tilde z) = \frac{c^2_e}{z^2_l} \xi(z), \; \; \, \tilde \zeta (\tilde z) = \zeta(z).\ee

Assuming \eq{dka}, using \eq{dkb}, \eq{dkc}, \eq{cjc}, transformations $\calt_{q, l, c_e} $ are reversible
\be\label{ube} \calt^{-1}_{q,l,c_e} = \calt_{q,l,c_e}.\ee

The symmetry groups $\{ (Z,\om)\}_0 $ for $\Si(z)$ and $\{(\tilde Z, \tilde \om)\}_0$ for $\tilde \Si(\tilde z)$, related by \eq{dja}, are related using \eq{dkb}, \eq{dkc}, \eq{cjc} and a condition implied by \eq{dka}, \eq{dkc},
\be\label{ubc} z^*_l = 0 = \tilde z^*_l.\ee

From \eq{cj}, \eq{cu}, \eq{ba}, for $\tilde \Si (\tilde z)$, using \eq{ubc},
\be\label{ubi} (\tilde \zeta (\tilde z))^* = \suml_{j \neq l} \tilde \zeta_{\tilde z_j} (\tilde z) \Big(\suml^n_{k = 1} \tilde Z_{jk} \tilde z_k + \tilde \om_j\Big).\ee

From \eq{cjc}, using \eq{dkb}, \eq{dkc},
\be\label{uba} \tilde \zeta_{\tilde z_j} (\tilde z) = \zeta_{z_j} (z) \frac{z_l}{c_e} = \zeta_{z_j} (z) \frac{c_e}{\tilde z_l}, \; \; \, j \neq l.\ee

Thus from \eq{ubi}, \eq{uba}, \eq{ubc}, for any $(Z, \om) \in \{ (Z, \om)\}_0$ satisfying \eq{dke} and
\be\label{ubb} Z_{jl}= 0, \; \; \, j = 1, \cdots, n,\ee
there exists $(\tilde Z, \tilde \om) \in \{ ( \tilde Z, \tilde \om)\}_0$ satisfying \eq{dke} and
\be\label{ubh} \tilde Z_{jk} = Z_{jk}, \; \; \, j, k = 1, \cdots, n, \; \; k \neq l;\ee
\be\label{ubd} \tilde Z_{jl} = \frac{\om_j}{c_e}, \; \; \, \tilde \om_j = 0, \; \; j = 1, \cdots, n.\ee

From reversibility \eq{ube}, any $(Z, \om)$ satisfying \eq{dke} and
\be\label{ubf} \om_j = 0, \; \; j = 1,\cdots, n\ee
determines $(\tilde Z, \tilde \om)$ satisfying \eq{dke}, \eq{ubh} and
\be\label{ubg} \tilde Z_{jl} = 0, \; \, \tilde \om_j = c_e Z_{jl}, \; \; j = 1, \cdots, n.\ee

\subsection{Examples of $Z$-systems}

$Z$-systems satisfying \eq{ca}, \eq{cc} may be readily constructed with certain prescribed nontrivial $\cals_R$.

\begin{thm} For given
\be\label{hba} W \in M^{\ntn}, \; \; Y \in \bbr^n, \; \;W Y= 0,     \;  W \hbox{\ symmetric} ,\ee
there exists a $Z$-system $\Sigma (z)$ satisfying \eq{ca}, \eq{cc} with nontrivial $\{ (Z,\om)\}_0$ such that using \eqref{ba}, for all
$(Z,\om) \in \{ (Z, \om)\}_0$
\be\label{hb} (Y\cdot z + \half (z \cdot Wz))^*= 0.\ee
\end{thm}
\begin{rema*} Either $Y$ or $W$ may vanish. Familiar  rotation symmetries are here included.
\end{rema*}
\begin{proof} Denote
\be\label{cy} D^\perp_{WY} \, \eqadef \{ y \in \bbr^n \, | \, Wy = 0, Y\cdot y = 0\}.\ee

Choose
\be\label{cw} \hat Y \in C^2(D \cap D^\perp_{WY} \to \bbr)\ee
otherwise arbitrary, and construct $\Sigma (z)$ from \eq{cc}, \eq{bt} and smooth function $\zeta(z), \xi(z)$
\be\label{ha} \zeta(z) = \zeta(\hat Y(z), \; \, Y\cdot z + \half z\cdot Wz), \; \; \xi(z) = \xi(\hat Y(z), \; \; Y\cdot z + \half z\cdot Wz).\ee

Using \eq{cy}, there exists a nontrivial space
\be\label{he} \{ Z\}_{WY} \, \subseteq\, \{ Z \in M^{\ntn} \, | \, WZ \hbox{\ antisymmetric, \ } Z^\dag y = 0\; \, y \in D^\perp_{WY}\},\ee
and such that for all $Z \in \{ Z\}_{WY}$,
\be\label{hez} Z^\dag Y \in \hbox{\ range\ }\; \,  W.\ee

For each $Z \in \{ Z\}_{WY}$, using \eq{hba}, \eq{cy}, \eq{hez}  there exists a unique $\om_Z\in \bbr^n$ satisfying
\be\label{hea} Y \cdot \om_Z = 0
\ee
\be\label{heb} y \cdot \om_Z = 0, \; \;  y \in D^\perp_{WY}
\ee
\be\label{hec} W\om_Z = - Z^\dag Y.
\ee

From \eq{cw}
\be\label{cwa} \hat Y_z (y)^\dag \in D^\perp_{WY}, \; \; \, y \in D,
\ee
and thus from \eq{cw}, \eq{ba}, \eq{he}, \eq{heb}, for all $Z \in \{ Z\}_{WY}$,
\be\label{hdb} \hat Y_z (y) (Zy+\om_Z) = 0,\; \;  y \in D.
\ee

Using \eqref{he}, \eq{hea}, \eq{hec}, for all $Z \in \{ Z\}_{WY}$,
\be\label{hdz} (Y^\dag + y^\dag W) (Zy + \om_Z) = 0
\ee
establishing \eqref{hb} for all $(Z, \om_Z), \; Z \in \{ Z\}_{WY},   \om_Z, y \in D$.

Thus from \eq{ha}, \eq{hdb}, \eq{hdz}, \eq{cu}
\be\label{hd} \{ (Z,\om_Z), Z \in \{ Z\}_{WY}\} \subseteq \{ (Z,\om)\}_0
\ee
for the so constructed system $\Sigma(z)$.
\end{proof}

\begin{cor} In the special case that
\be\label{hdx} D^\perp_{WY} = \{ 0\},\ee
 the condition \eq{hd} holds with equality.
 \end{cor}

In this special case, by appeal to \eq{cd}, it suffices to consider the specific form of $\zeta(z)$
\be\label{hdy} \zeta(z) = Y\cdot z + \half z\cdot Wz + \hbox{\ constant}.
\ee

 \subsection{Richness of the symmetry group}

More generally, for $Z$-systems $\Sigma(z)$ we determine disjoint subspaces
\be\label{hh}  \land_\Sigma \oplus \land^\perp_\Sigma = \bbr^n
\ee
such that
\be\label{hi} \land^\perp_\Si \, \eqadef\, \{ e \in \bbr^n \, | \, Ze = Z^\dag e = \om \cdot e = 0\}
\ee
for all $(Z, \om)$ such that
\begin{align}\label{hha} \zeta^*(z) & = \zeta_z(z) (Zz+\om)\nonumber \\
&= c_{\zeta, Z}
\end{align}
follows from \eq{ba}, \eq{bb}, \eq{cj}.

For $e_\perp \in \land^\perp_\Sigma, \; (Z, \om)$ satisfying \eq{hha}, using \eq{ba}, \eq{hi}
\begin{align}\label{hj} (e_\perp\cdot z)^* &= e_\perp\cdot(Zz+\om)\nonumber\\
&=0,
\end{align}
and using \eq{ch}, \eq{hi}
\begin{align}\label{hja} (e_\perp\cdot x)^* &= - e_\perp\cdot Z^\dag x\nonumber\\
&=0.\end{align}

The components $e_\perp\cdot z,\; \;  e_\perp  \cdot x$ are invariant for all $(Z,\om)$ satisfying \eqref{hha}.  Additionally, $z^*$ is independent of $e_\perp\cdot z$ and $x^*$ is independent of $e_\perp\cdot x$.

In \eq{hh}, we introduce additional subspaces
\be\label{sza} \Lambda_\Si \, \eqadef \, \Lambda^V_\Si \oplus \La^I_\Si;\ee
\be\label{szb} \La^V_\Si \, \eqadef\, \{ e \in \bbr^n \, | \, Z^\dag e\neq 0 \hbox{ or\ } \om \cdot e \neq 0 \} \ee
for some $(Z, \om)$ satisfying \eq{hha};
\be\label{szc} \La^I_\Si \, \eqadef\, \{ e \in \bbr^n\, | \, Z^\dag e = \om \cdot e = 0 \} \ee
for all $(Z,\om)$ satisfying \eq{hha}, but
\be\label{szd} Z e \neq 0 \ee
for some such $Z$.

Using \eq{hh}, \eq{hi}, \eq{sza}, \eq{szb}, \eq{szc}, uniquely for any $y \in D$,
\begin{align}\label{sd} y &= y_V + y_I + y_\perp; \; \; \, y_V = e_V (e_V\cdot  y), \; \, e_V, y_V \in \La^V_\Si;\nonumber \\
y_I  & = e_I (e_I \cdot y), \; \; e_I, y_I \in \La^I_\Si; \; \, y_\perp = e_\perp (e_\perp \cdot y), \; \; e_\perp , y_\perp \in \La^\perp_\Si.\end{align}

Richness of the symmetry group $\cals_R$ is then quantified by an integer $L$, depending on $\Si(z)$
\begin{align}\label{gab} L &\eqadef \, \dim  \La_\Si\nonumber\\
&= \dim \La^V_\Si + \dim \La^I_\Si\nonumber \\
&= n - \dim \La^\perp_\Si\end{align}

\begin{lem}  Assume a $Z$-system satisfying \eq{ca}, \eq{cc}, with $\zeta(z)$ of the form \eq{ha}, as determined using \eq{hba}, \eq{cy}, \eq{cw}.  Assume $\xi(z)$ either of the same form or else determined from an equation of state \eq{me} satisfying \eq{mf}, \eq{mga} below.

Then with $\land^\perp_\Si$ determined from \eq{hi}, \eq{hha},
\be\label{hl} \land^\perp_\Si\subseteq D^\perp_{WY},\ee
with equality if and only if
\be\label{hla} D^\perp_{WY} \subseteq \ker Z, \; \; \om\perp D^\perp_{WY}\ee
for all $Z\in \{ Z\}_{WY}, \, \om = \om_Z.$
\end{lem}
\begin{proof} Denote
\be\label{hlb} \tilde \land^\perp_\Si\, \eqadef\, \{ y \in D\, | \, y \cdot (Zy'+\om) = 0 \}\ee
for all $y'\in D, \; (Z,\om)$ satisfying \eq{hha}.

 Using \eq{hd} in \eq{hlb}, then \eq{hea}, \eq{heb}, \eq{hec},
 \be\label{hlc} \tilde \land^\perp_\Si = D^\perp_{WY},\ee
 and \eq{hl} follows from comparison of \eq{hlb} with \eq{hi}. If \eq{hla} holds, the two spaces coincide.
 \end{proof}

 \begin{defn} A $Z$-system is closed if $\zeta(z) $ (but not $\xi (z)$) is homogeneous of degree zero.
 \end{defn}

 For a closed $Z$-system, from \eqref{cc},
 \be\label{ma} \suml^n_{j = 1} z_j\psi_j (z) = 0.\ee

 From \eqref{ma}, \eqref{bt}
 \begin{align}\label{mb} 0 &= \suml^n_{k = 1} z_k (\suml^n_{j=1} z_j\psi_j (z))_{z_k}\nonumber \\
 &= \suml^n_{k = 1} (z_k \psi_k (z) + z_k \suml^n_{j=1} z_j\psi_{j, z _k}  (z))\nonumber \\
 &= \suml^n_{j, k = 1} z_k z_j a_{kj}(z).\end{align}

 Thus using \eqref{ma}, \eqref{mb} in \eqref{bs}, the entropy flux $q$ satisfies
 \be\label{mc} \suml^n_{j = 1} z_j q_j (a_{j \cdot}) = 0.\ee

\subsection{Equation of state}

In determination of a $Z$-system of dimension $n$ from \eq{cc}, \eq{bt},  an alternative to independent specification of $\xi(z), \; \zeta(z)$ is explicit specification of $\zeta(z)$ and an equation of state implicitly determining $\xi(z)$ from $\zeta(z)$.
Such systems are of primary interest below, as an equation of state will be shown to necessarily hold for systems with a sufficiently rich symmetry group.

 \begin{defn}   A $Z$-system satisfying \eqref{cc} is associated with an equation of state if there exists a scalar function
 \be\label{md} z_\zeta \in C^1 (D \to \bbr^{d_\sigma})\ee
 such that $\xi(z), \zeta(z)$ are related by
 \be\label{me} \zeta(z) = \si(\xi(z), z_\zeta(z))\ee
 with $\si $ of class $C^1 $ satisfying
\be\label{mf} \frac{\partial\si}{\partial\xi} \, \mathop{\midl}_{z_\zeta} > 0.\ee
\end{defn}

In any specific example, $z_\zeta$ may be absent in \eqref{me}.  It is assumed throughout that any $z_\zeta$ in \eq{me} satisfies
\begin{align}\label{mga} z^*_\zeta (y)  &= z_{\zeta,z} (y)\; \; (Zy+\om)\nonumber \\
&= 0 \end{align}
for all $y \in D$,  for all $(Z, \om)$ such that \eqref{hha} holds.

Then from \eqref{me}, using \eq{hha}, \eq{mga}, for all $(Z,\om)$ such that \eq{hha} holds,
\begin{align}\label{mh} 0 &= \zeta(z)^*\nonumber \\
&=  \zeta_\xi (\xi, z_\zeta) \xi^* + \zeta_{z_\zeta} (\xi, z_\zeta) z^*_\zeta\nonumber\\
&= \zeta_\xi (\xi, z_\zeta) \xi^*,\end{align}
and
\be\label{mi} \xi^* = 0\ee
follows from \eqref{mh}, \eq{mf}.  Thus \eq{cu} holds, and $\{ (Z,\om)\}_0$ is the space of $(Z, \om)$ satisfying \eq{hha}.

Sufficient conditions for \eq{me}, \eq{mf} are obtained using  \eq{sd}.  We consider trajectories within $\Om$ determined from
\be\label{saa} y_\tau (\tau) = Zy (\tau) + \om, \; \; \, \tau \in \bbr,\ee
with $(Z,\om)$ satisfying \eq{hha} and initial data
\be\label{sab} y(0) = y_V (0) + y_I (0) + y_\perp (0)\ee
determining $e_V, e_I, e_\perp$ from \eq{sd}.

Using \eq{hi}, \eq{szc}, \eq{szd}, solutions of \eq{saa}, \eq{sab} satisfy
\be\label{sac} y_{V,\tau} (\tau) = e_V (Z^\dag e_V \cdot y_V (\tau) + e_V \cdot Zy_I (0) + e_V\cdot \om)\ee
\be\label{sad} y_{I,\tau} (\tau) = y_{\perp, \tau} (\tau) = 0.\ee

Within such trajectories, for $(Z,\om) \in \{ (Z,\om)\}_0$, from \eq{cu}
\be\label{sba} \zeta(y(\cdot))_\tau = \xi(y(\cdot))_\tau = 0 \ee
while for $(Z, \om) \in \{ (Z, \om)\}_\zeta$ in \eq{cxa}, from \eq{cxb}
\be\label{sbb} \zeta(y(\cdot))_\tau > 0.\ee

\begin{thm} Assume
\be\label{sb} \dim \{ ( Z, \om)\}_\zeta = 1\ee
in \eq{cx}, \eq{cxa}, \eq{cxb}.

For arbitrary $y_I \in \La^I_\Si, \; y_\perp \in \La^\perp_\Si$, assume $y_V (0) \in \La^V_\Si$ (depending on $y_I, y_\perp$) such that 1-manifolds $\Om_1(y_I, y_\perp)$ obtained from \eq{sab}, \eq{sac}, \eq{sad} with $Z = Z_\zeta, \; \om = \om_\zeta$ continue to $\partial D$ in both directions with respect to $\tau$.

Assume the function $\xi(z)$ such that
\be\label{sbc} \xi_z (y(\tau)) (Z_\zeta y(\tau) + \om_\zeta) > 0\ee
within $\om_1 (y_I, y_\perp)$.

Assume in addition that any point $y = (y_V, y_I, y_\perp)$ may be connected to some point in $\Om_1 (y_I, y_\perp)$ by a sequence of trajectories \eq{saa}, equivalently \eq{sac}, \eq{sad}, with $(Z, \om) \in \{ (Z, \om)\}_0$.

Then \eqref{me}, \eq{mf} hold throughout $D$, with
\be\label{sc} d_\sigma \le \dim \La^I_\Si + \dim \La^\perp_\Si.\ee
\end{thm}
 \begin{proof} Within each $\Om_1(y_I, y_\perp)$, from \eq{sbb}, \eq{sbc} we have $\zeta(y(\tau)), \; \xi (y(\tau))$ increasing monotonically with respect to  $\tau$.  Thus determining $z_\zeta$ as required using the values of $y_I, y_\perp$, we have \eq{me}, \eq{mf} within each $\Om_1 (y_I, y_\perp)$ satisfying \eq{sc}.

 By assumption, any point $(y_V,  y_I, y_\perp)$ is connected to some $y_V (\tau), y_I, y_\perp) \in \Om_1(y_I, y_\perp)$ by trajectories \eq{saa} with $(Z,\om) \in \{ (Z, \om)\}_0$.

 Using \eq{sba}
 \begin{align}\label{se} \zeta(y_V, y_I, y_\perp) &= \zeta(y_V(\tau), y_I, y_\perp)\nonumber\\
 \xi(y_V, y_I, y_\perp) &= \xi(y_V(\tau), y_I, y_\perp).\end{align}

 The value of $\tau$ such that \eq{se} holds is unique, using \eq{sbb}.  Thus from \eq{se}, the relationship
 \eq{me}, \eq{mf} holds throughout $D$.
 \end{proof}

In the special case of $\Si(z)$ satisfying \eq{ca}, \eq{cc} and \eq{me} of the form
\be\label{mea} \sigma = \si(\xi)\ee
with $z_\zeta(z) $ absent, such that \eqref{mf} holds, we introduce additional transformations $\calt_{\zeta, f}$.

The image
\be\label{mia} \Si'(z) = \calt_{\zeta, f} \Si(z)\ee
is also of the form \eq{ca}, \eq{cc}, $z$ unchanged, with
\be\label{mib} \zeta' (z) = f(\zeta(z))\ee
for functions
\be\label{mic} f \in C^2 (\{ \zeta (y), \; \, y \in D\} \to \bbr), \; \; \, f_\zeta (\cdot) \neq 0.\ee

The system $\Si'(z)$ is associated with an equation of state
\be\label{mid} \zeta' (z) = \sigma'(\xi' (z))\ee
with $\sigma'$ satisfying
\be\label{mie} \sigma' \Big(\frac{\sigma^{-1} (\la)}{f_\zeta(\la)}\Big) = f (\la), \; \; \, \la \in \{ \zeta(y), \; \, y \in D \}.\ee

From \eq{mib}, \eq{mid}, \eq{mie}, $\xi  (z), \xi'(z)$ are related by \eq{cd}, and from \eq{cc}, \eq{bt}
\be\label{mih} \psi' (z) = \psi (z), \; \; \, a'_{\cdot, \cdot} (z') = a_{\cdot, \cdot} (z) \ee

From \eq{mie}, \eq{mic}, the condition \eq{mf} and thus \eq{mi} extends to $\Si'(z)$, so the nontrivial symmetry group is preserved,
\be\label{mig} \{ (Z', \om')\}_0 = \{ (Z,\om)\}_0.\ee

\subsection{Hyperbolicity}

Sufficient conditions for hyperbolicity \eq{btc} are readily obtained for $Z$-systems $\Si(z)$  satisfying \eq{ca}, \eq{cc}, with an associated equation of state \eq{me} satisfying \eq{mf}, \eq{mga}.

The corresponding symmetry group $\{ (Z, \om)\}_0$ is first used to remove ambiguity  in a time-like direction $e_H$ \eq{btb}.

Necessarily $e_H \cdot x $ is invariant, requiring
\be\label{btd} X^\dag e_H = Ze_H = 0 \ee
for all $(Z, \om) \in \{ ( Z,\om)\}_0$, using \eq{ch}.  However $Z^\dag e_H $ need not vanish.

Setting
\be\label{btm}\theta^* = Z\theta, \; \; (Z,\om) \in \{ Z, \om)\}_0\ee
by convention, using \eq{bc}, \eq{bu} \eq{ch}, \eq{btd}, we obtain
\be\label{btl} (\theta \cdot (e_H \cdot \psi)_{zz} (z) \theta)^* = 0 \ee
for all $z\in D, \; (Z, \om) \in \{ ( Z, \om)\}_0$.  Thus hyperbolicity of a $Z$-system is invariant under $\cals_R$.

\begin{defn} A $Z$-system $\Si(z)$ admits a time-like direction $e_H $ if \eq{btb}, \eq{bta}, \eq{btc}, \eq{btd} hold.
\end{defn}

We observe also that \eq{btd} is necessary for reduced form of $\Si(z)$ self-similar with respect to $e_H \cdot x$.

\begin{defn} A $Z$-system is hyperbolic if there exists a time-like direction $e_H$, as per definition 2.11, such that \eq{btc} holds. A $Z$-system is uniformly hyperbolic if \eq{btc} holds with strict inequality throughout $D$.
\end{defn}

Explicit sufficient conditions for hyperbolicity are now obtained.

Using \eq{cc}, \eq{me} with
\be\label{bwa}z_\zeta(z) = e_{z_\zeta} \cdot z,\; \; \, e_{z_\zeta} \in \bbr^n\ee
in \eqref{btc}, we obtain
\begin{align}\label{bwb}
&(e_H \cdot \psi)_{zz} = \big(\xi\zeta_{(e_H\cdot z)} \big)_{zz}\nonumber\\
= &\xi\zeta_{(e_H\cdot z)zz} + \frac{ \zeta_{(e_H\cdot z)}\zeta_{zz}}{\sigma_\xi} \nonumber \\
&+ \frac{1}{\sigma_\xi} \big( \zeta^\dag_{(e_H \cdot z)z} \zeta_z + \zeta^\dag_z \zeta_{(e_H\cdot z)z}\big)\nonumber\\
&- \frac{ \zeta_{(e_H\cdot z)} \sigma_{\xi\xi} }{ \sigma^3_\xi }
 (\zeta^\dag_z - \sigma_{z_\zeta} e_{z_\zeta}) \big(\zeta_z - \sigma_{z_\zeta} e^\dag_{z_\zeta}\big) \nonumber \\
&+ \frac{1}{\sigma_\xi } (\zeta^\dag_{(e_H\cdot z) z} \zeta_z + \zeta^\dag_z \zeta_{(e_H \cdot z)z})\nonumber \\
&-\frac{\sigma_{z_\zeta}}{\sigma_\xi } (\zeta^\dag_{(e_H\cdot z )z} e^\dag_{z_\zeta} + e_{z_\zeta} \zeta_{(e_H\cdot z)z})\nonumber \\
&- \frac{\zeta_{(e_H\cdot z)}}{\sigma_\xi} \sigma_{z_\zeta z_\zeta} e_{z_\zeta} e^\dag_{z_{\zeta}}.\end{align}

Sufficient conditions for \eq{btc} are that the separate terms in \eq{bwb} are each nonnegative semidefinite, in the sense of symmetric $n$-matrices.

For the special case of an equation of state of the form \eq{mea}, such follows from
\be\label{bwc} \sigma_{\xi\xi} < 0 < \xi, \sigma_\xi,  \zeta_{(e_H\cdot z)},\ee
\be\label{bwg} \zeta_{(e_H\cdot z)zz} \ge 0,\ee
\be\label{bwd} \zeta^\dag_{(e_H\cdot z)z} \zeta_z + \zeta^\dag_z \zeta_{(e_H\cdot z)z} + \zeta_{(e_H\cdot z)} \zeta_{zz}
 - \frac{\sigma_{\xi\xi}}{\sigma^2_\xi} \zeta_{(e_H \cdot z)} \zeta^\dag_z \zeta_z  \ge 0.\ee

In the more general case \eq{me}, \eq{bwa}, additional conditions are required:
\be\label{bwe} \zeta^\dag_{(e_H\cdot z)z} \| e_{z_\zeta} , \; \, \sigma_{z_\zeta} \; (\zeta_{(e_H\cdot z)z} e_{z_\zeta}) \le 0;\ee
\be\label{bwf} \sigma_{z_\zeta z_\zeta} \le 0.\ee

\subsection{Euler Systems}

Euler systems, models of fluid flow, provide motivating examples for this study.  The isentropic system is a $Z$-system, satisfying \eq{bt} with
\be\label{eaa} z = \begin{pmatrix} H(P) - \tfrac 12 |u|^2\\ u_1\\ \vdots\\ u_{m-1} \end{pmatrix} , \; \; \; \psi(z) = P \; \begin{pmatrix} 1\\ u_1\\ \vdots\\ u_{m-1} \end{pmatrix} . \ee

In \eq{eaa}, $u_i$ are the fluid velocity components, $P$ the pressure, and $H(P) $ the enthalpy.

The condition \eq{cc} is satisfied with
\be\label{eab} \xi (z) = P, \; \; \zeta(z) = z_1 + \tfrac 12 \suml^m_{j = 2} z^2_j\ee
of the form \eq{hdy}, related by an equation of state
\be\label{eac} \zeta(z) = H(P) = \sigma(\xi)\ee
of the form \eq{mea}.

In particular, for this class of systems \eq{sb}, \eq{cxa}, \eq{cxb} hold with
\be\label{sl} Z_\zeta = 0, \; \; \om_\zeta = \begin{pmatrix} 1\\ 0\\ \vdots\\ 0 \end{pmatrix}
\ee
and \eq{bwc} suffices for hyperbolicity.

For this system, the symmetry group $\{ (Z, \om)\}_0$ is as constructed in \eq{he}, \eq{hez}, \eq{hea}, \eq{heb}, \eq{hec}, with $\hat Y (z), \; D^\perp_{WY}$ vanishing.

For this class of systems, \eq{gab} holds with
\be\label{so} L = n,\ee
and \eq{sza} with
\be\label{sm} \dim \La^I_\Si = 0.\ee

The time-like direction $e_H$ satisfies \eq{btd} with
\be\label{sp} e_H = \hat e_1,\ee
with $\hat e_k, \; \; k = 1, \cdots, n$ denoting unit vectors throughout.

By inspection, this class of systems is a hierarchy with respect to $n$, determined from transformations $\calt_{\hat e_k, 0}, \; \, k = 2, \cdots, n$, as introduced in \eqref{cjd}.

The Euler system including conservation of energy is of the form \eq{aa}, \eq{ab}, \eq{ac}, \eq{bt} with
\be\label{eba} n = m+1,\ee

\be\label{ebb} z' = \frac 1T \begin{pmatrix} &G(P,T) - \tfrac 12 |u|^2\\ &\; \; \, u_1\\ &\; \; \, \vdots \\ &\; \; \, u_{m-1} \\ &-1 \end{pmatrix} , \; \; \, \psi' (z') = \frac PT \begin{pmatrix}
1 \\ u_1\\ \vdots \\ u_{m-1} \end{pmatrix} \ee
with $T$ the fluid temperature and $G(P,T)$ the Gibbs free energy.

For this system, a truncated form of \eq{cc} holds,
\be\label{ebc} \psi'_i (z') = \xi'(z') \zeta'_{z_i} (z'), \; \; i = 1,\cdots, m\ee
with
\be\label{ebd} \xi'(z') = \frac{P}{T^2}, \; \; \zeta' (z') = - \frac{z'_1}{z'_n} + \tfrac 12 \suml^m_{j = 2} \;\, \frac{z'^2_j}{z'^2_n} ,\ee
related by an equation of state
\be\label{ebe} \zeta'(z') = G(P,T) = \sigma' (\xi', z'_n).\ee

From \eq{ebd}, this system is closed as per definition 2.8.

This system may be regarded as the reduced form of an extended Euler system, a $Z$-system of dimension $n$, which is stationary with respect to an additional independent variable $\nu = x_n$, which has physical dimensions of length$^2/\hbox{time}$.  The corresponding potential function
\be\label{eca} \psi'_n (z') = \frac PT  (G(P,T) + \tfrac 12 |u|^2)\ee
is obtained by extending \eq{ebc} to $i = n$, using \eq{ebd}, thus recovering \eq{cc}.

For the extended Euler system, by convention
\be\label{ecb} x = \begin{pmatrix} t\\ x_1\\ \vdots\\ x_{m-1} \\ \nu\end{pmatrix} .\ee

Then from \eqref{bt}, \eq{ebb}, \eq{eca}, this system is
\begin{align}\label{ecc}
 &\rho_t + \suml^{m-1}_{i = 1} (\rho u_i)_{x_i} + (\rho(G + \tfrac 12 |u|^2))_\nu = 0;\\
 &(\rho u_i)_t + (\rho u^2_i + P)_{x_i} + \suml_{j\neq i} (\rho u_i u_j)_{x_j}\nonumber\\
\label{ecd} &\qquad \; \; + (\rho u_i(G + \tfrac 12 |u|^2) + 2 u_iP)_\nu = 0, \; \; i = 1, \cdots, m - 1;\\
&(\rho(E + \tfrac 12 |u|^2))_t + \suml^{m-1}_{i = 1} (\rho u_i(H + \tfrac 12 |u|^2))_{x_i}\nonumber \\
\label{ece} & \qquad \; \; + (\rho (E + \tfrac 12 |u|^2) (G + \tfrac 12 |u|^2) + 2 | u|^2 P)_\nu = 0.\end{align}

In \eq{ecc}, \eq{ecd}, \eq{ece}, the symbols admit familiar interpretation: $\rho$ is the fluid density; $E = E(V, S)$ the internal energy, $V$ the specific volume, and $S$ the specific entropy; $H = H(P,S)$ the enthalpy.  The same conditions on the equation of state for hyperbolicity apply as for the Euler system including conservation of energy.

For the system \eq{ecc}, \eq{ecd}, \eq{ece}, the corresponding entropy flux $q'$ is
\begin{align}\label{ecf}
q' &= \rho S \begin{pmatrix} 1 \\ u_1\\ \vdots\\
u_{m-1}\\ G+\tfrac 12 |u^2| \end{pmatrix} \nonumber\\
&=\frac{TS}{PV} \; \psi' (z').
\end{align}

The conditions \eq{so}, \eq{sp} apply also to this class of systems.  The condition \eq{sza} holds with
\be\label{sq} \La^I_\Si = span \{ \hat e_n\}\ee
and \eq{sb}, \eq{cxa}, \eq{cxb} with
\be\label{sr} Z_{\zeta, 1, n} = 1, \; \, \om_\zeta = 0,\ee
all other components of $Z_\zeta $ vanishing.

In view of \eqref{ebe}, the conditions \eq{bwe}, \eq{bwf} are additionally required for hyperbolicity.

The hierarchy of systems with respect to $n$ is now determined using $\calt_{\hat e_k, 0}, \; \, k = 2, \cdots, n-1$.

The isentropic Euler system $\Si(z)$ and the extended Euler system $\Si'(z')$, $Z$-systems of dimension $n-1, n$, respectively, are related by transformations described above.  In particular
\be\label{eda} \Si(z) = \calt_{\hat e_n, -1} \Si'(z');\ee
the value
\be\label{edb} c_e = - 1\ee
in \eq{cjd} is determined from the adopted form of $z'_n$ in \eq{ebb}.

Conversely
\be\label{edc} \Si' (z') = \calt_{q, n, c_e} \hat \Si(\hat z)\ee
as determined from \eq{dja}, \eq{dkb}, \eq{dkc}, \eq{cjc}, \eq{edb} with $l = n$, with $\hat \Si(\hat z)$ a $Z$-system of dimension $n$ in which entropy rather than energy is conserved.

This system is determined using \eq{cc} with
$\xi(z) = \hat \xi (\hat z), \zeta(z) = \hat \zeta (\hat z)$ in \eqref{eab},
\be\label{edd} \hat z = \begin{pmatrix} z_1\\ \cdots\\ z_m\\ \hat z_n\end{pmatrix} \; , \; \, \hat z_n = T\ee
using \eqref{eba}.  The equation of state for $\hat \Si (\hat z)$ is
\be\label{ede} \hat \zeta (\hat z) = G(P, T) = \hat\sigma (\hat \xi(\hat z), \hat z_n);\ee
the extension of an isentropic system $\Si(z)$ to $\hat\Si(\hat z)$ occurs only through the equation of state.  We note that from \eq{eab}, \eq{edd}, the condition \eq{cga} fails, with $e_\perp  = \hat e_n$.

\section{Coupling of $Z$-systems}

Given independent systems $\Si^k (z^k), \; \, k = 1,\cdots, K$, each of the form \eq{aa}, \eq{ab}, \eq{ac}, \eq{bt}, sharing a common domain $\Om$ but with
\be\label{laz} z^k(x) \in D^k \subseteq \bbr^{n_k}, \; \; \, k = 1, \cdots, K,\ee
may be coupled by imposition of a relationship on the respective dependent variables $z^k$.  The effect of such on the respective symmetry groups  $\{ (Z^k, \om^k, X^k)\}$ satisfying \eq{ba}, \eq{bu} for each value of $k$ is addressed.  Attention is restricted to the case
\be\label{laa} c_Z = 0 \ee
in \eq{bu}, independently of $k$.

The given systems are conveniently assembled into a single system, also of the form \eq{aa}, \eq{ab}, \eq{ac},
\be\label{lad} \bar\Si(\bar z) \, \eqadef\, (\Si^1 (z^1) \cdots \Si^K (z^K))\ee
with the same domain $\Om$ but of dimension
\be\label{lab} \bar n \, \eqadef\, n + 1 \, \eqadef\, \suml^K_{k = 1} \, n_k,\ee
with dependent variable
\be\label{lac} \bar z \, \eqadef\, \begin{pmatrix} z^1 \\ \vdots \\ z^K\end{pmatrix}.\ee

The system $\bar \Si (\bar z) $ also satisfies \eq{bt}, with
\be\label{lae}\bar \psi (\bar z) \,\eqadef \,\suml^K_{k = 1} \psi^k(z^k).\ee

From \eq{lac}, \eq{lae}, the symmetry group $\{( \bar Z, \bar\om, \bar X)\}$ for the system $\bar \Si(\bar z)$ satisfies \eq{ba}, \eq{bu}, \eq{laa} with
\be\label{laf} \bar Z = \hbox{ diag\ } (Z^1, Z^2, \cdots, Z^k) \in M^{\bar n \times \bar n}\ee
\be\label{lag} \bar \om = \begin{pmatrix} \om^1\\ \vdots\\ w^K\end{pmatrix} \, \in \bbr^{\bar n} \ee
for $(Z^k, \om^k, X^k), \;\;  k = 1, \cdots, K$, restricted by
\be\label{lah} X^k = \bar X, \; \;  \, k = 1, \cdots, K\ee
because of the common domain $\Om$.

A coupled system $\Si(z)$, also of the form \eq{aa}, \eq{ab}, \eq{ac}, in the same domain $\Om$ but of dimension $n$ in \eq{lab}, is obtained by a simplification of the form \eq{dh},
\be\label{lba} e_\la \cdot \bar z = c_\la, \; \; \, e_\la \in \bbs^n, \; \; \, c_\la \in \bbr,\ee
applied to the system $\bar \Si (\bar z)$; with parameter(s) $\la$ to be determined below.

The dependent variables $z, \bar z$ are related by
\be\label{lbc} z = \Ga_\la \bar z, \; \; \, \Ga_\la \in M^{n\times \bar n}\ee
necessarily satisfying
\be\label{lbd} \Ga_\la e_\la = 0 \ee
and such that
\begin{align}\label{lbf} &\bar \Ga_\la \in M^{\bar n \times \bar n}\nonumber \\
&\bar \Ga_{\la, ij} \, \eqadef \, \begin{cases}  \Ga_{\la, ij},  \; \; \, i = 1, \cdots, n, \; \; \, j = 1, \cdots, \bar n\\
e_{\la, j}, \; \; i = \bar n, \; \; \, j = 1, \cdots, \bar n\end{cases}\end{align}
is orthogonal.  Thus from \eq{lbc}, \eq{lbd}, \eq{lbf}
\be\label{lbh} {z \choose c_\la} = \bar \Ga_\la \bar z, \; \; \, \bar z = \bar\Ga^\dag_\la \; {z\choose c_\la} \ee
and
\be\label{lbi} \frac{\partial\bar z}{\partial z} = \Ga^\dag_\la.\ee

The coupled  system $\Si(z)$ also satisfies \eq{bt}, with
\be\label{lbe} \psi(z) \, \eqadef\, \bar \psi(\bar z).\ee

Then from \eq{lbc}, \eq{lbe}, \eq{lae}, \eq{lad},
\be\label{lbg} \Si (z) = \bar \Si (\bar z) \Ga^\dag_\la\ee
with density/flux functions related by
\be\label{lbj} a_{i \cdot} (z) = \Ga_\la \bar a_{i\cdot} (\bar z), \; \; \, i = 1, \cdots, m.\ee

Given hyperbolic constituent systems $\Si^k(z^k)$ as per \eq{btb}, \eq{btc}, this coupling strategy determines a hyperbolic system $\Si(z)$.  Without loss of generality, applying transformations $\calt_{QU}$ \eq{an}, \eq{anb} to each $\Si^k(z^k)$ as necessary, we assume common $e_H$,
\be\label{lja} (e_H \cdot \psi^k)_{z^kz^k} (z^k) \ge 0, \; \; \, k = 1,\cdots, K.\ee

Then from \eq{lae}, \eq{lac}, \eq{lja}
\be\label{ljb} (e_H \cdot \bar \psi)_{\bar z \bar z} (\bar z) \ge 0 \ee
and from \eq{ljb}, \eq{lbc}, \eq{lbe}, \eq{lbi}
\begin{align}\label{ljc} (e_H \cdot \psi)_{zz} (z) & = \Ga_\la ((e_H\cdot \bar \psi)_{\bar z \bar z} (\bar z)) \Ga^\dag_\la\nonumber \\
&\ge   0.\end{align}

The symmetry group for $\Si (z)$ is obtained from that for $\bar \Si (\bar z)$ by application of \eq{di}, \eq{dj}, \eq{djb}.

From \eqref{di}, \eq{lba}, using \eq{ba}, \eq{laf}, \eq{lag}, necessarily
\begin{align} \label{lda} 0 &= e_\la \cdot \bar z^*\nonumber\\
&=e_\la \cdot (\bar Z \bar z + \bar\om)\end{align}
for all $\bar z$, which can hold only if $\bar Z, \bar\om, e_\la$ satisfy
\be\label{ldc} \bar Z^\dag e_\la = 0 \ee
and
\be\label{ldd} \bar\omega\cdot e_\la = 0.\ee

Thus \eqref{dj}, \eq{djb} are recovered.  We note that \eq{ldc} permits
\be\label{lde} \bar Z e_\la \neq 0, \; \; \; e_\la \notin \land^\perp_{\bar\Si(\bar z)}.\ee

With $\bar Z, \bar \om, e_\la $ so restricted, to obtain a symmetry group $\{ (Z,\om, X)\}$ for $\Si(z)$, satisfying \eq{bu}, \eq{ba}, \eq{laa}, for each $(\bar Z, \bar \om, \bar X)$ it suffices to put
\be\label{lca} X = \bar X\ee
and determine the corresponding $Z,\om$ such that
\begin{align}\label{lcb} z^* \,  &\eqadef \, Zz+\om\nonumber\\
&= \Ga_\la \bar z^*\end{align}
using \eq{lbc}.

In \eq{lcb}, successively we use \eq{ba}, \eq{laf}, \eq{lab}, then \eq{lbh}, then \eq{lbf} to obtain
\begin{align}\label{lcc} z^* &= \Ga_\la (\bar Z\bar z + \bar \om)\nonumber \\
&= \Ga_\la (\bar Z \bar \Ga^\dag {z\choose c_\la} + \bar \om)\nonumber\\
&= \Ga_\la (\bar Z (\Ga^\dag z + c_\la e_\la) + \bar \om).\end{align}

Comparison of \eq{lcb}, \eq{lcc} identifies
\be\label{lcd} Z = \Ga_\la \bar Z \Ga^\dag_\la,\ee
\be\label{lcd} \om = \Ga_\la \bar \om + c_\la \Ga_\la \bar Z e_\la.\ee

Remaining unaddressed is the determination of $\bar Z, \bar \om, \bar X$ satisfying \eq{lah}, and of $e_\la$ satisfying \eq{ldc}, \eq{ldd}.

In the special case that the given systems $\Si^k(z^k)$ coincide, with
\be\label{lfa} \Si^k(z^k) = \Si^1(z^k), \; \; \; k = 2, \cdots, K,\ee
the sets $\{ (Z^k, w^k, X^k)\}$ will be independent of $k$.  Then
\be\label{lfb} \bar X = X^1\ee
suffices in \eq{lah}.  As in this case
\be\label{lfc} n_k = n_1, \; \; \, k = 2, \cdots, K,\ee
it suffices to choose
\be\label{lfd} e_\la = \begin{pmatrix} \la_1 \ue\\
\vdots \\
\la_k \ue\end{pmatrix}, \; \; \; \la \in \bbs^{K-1}, \; \; \, \ue \in \bbs^{n_1 -1}.\ee

Now \eq{ldc} follows from $\ue$ satisfying
\be\label{lfe} Z^{1\dag} \ue = 0,\ee
and \eq{ldd} follows from
\be\label{lff} \big( \suml^K_{k = 1} \la_k\big) (\om^1\cdot \ue) = 0\ee
retaining considerable flexibility in $\la$ for $K \ge 3$.

Similar results hold for the given $\Si^k(z^k)$ $Z$-systems with $\zeta_k$ coinciding,
\be\label{lga} \zeta_k(z^k) = \zeta_1 (z^k), \; \; \, k = 2, \cdots, K,\ee
differing only in the associated equations of state \eq{me}, satisfying \eq{mf}, \eq{mga}.

The spaces  $\{ ( Z^k, \om^k)\}_0$ as determined from \eq{cj}, \eq{cu} are independent of $k$, and \eq{lah} holds with
\be\label{lgb} \bar X = X^1 = - Z^{1\dag}\ee
using \eq{ch}.

The results \eq{lfd}, \eq{lfe}, \eq{lff} hold with
\be\label{lgc} n_k = m \ee
using \eq{ca}.

More generally, assume the given $\Si^k(z^k)$ to be $Z$-systems, satisfying \eq{lgc} with prescribed $\zeta_k (z^k)$ and equations of state \eq{me},
\be\label{lha} \zeta_k (z^k) = \sigma_k (\xi_k (z^k),  z^k_\zeta (z^k))\ee
satisfying \eq{mf}, \eq{mga}.

The spaces $\{ (Z^k, \om^k)\}_0$ determined from \eq{cj}, \eq{cu} separately for each $k = 1, \cdots, K$ generally will not coincide, but from \eq{ch}
\be\label{lhb} X^k = - Z^{k\dag}, \; \; \; k = 1, \cdots, K.\ee

Determination of $\bar X$ satisfying \eq{lah} is materially simplified by the fact that from theorem 2.4 and \eq{lhb}, each $\{ X^k\}$ is a linear vector space.  For each $\{ X^k\}$, basis elements $X^k_l, \; l = 1, \cdots $ are determined with coinciding characteristic and minimal polynomials $P^k_l (X^k_l)$. Typical expressions for $P^k_l (X^k_l)$ include
\[ (X^k_l)^2 + I_m, \; \; \, (X^k_l)^2, \; \; (X^k_l)^3.\]

A necessary condition for \eq{lah} is existence of $P^k_l$ independent of $k$, such that
\be\label{lhc} P^k_l (X^k_l) = P^1_l (X^k_l), \; \; k = 2, \cdots, K \ee
for some value(s) of $l$.

Then
\be\label{lhd}  X^k_l = X^1_l, \; \; \, k = 2,\cdots, K\ee
follows from similarity transformations \eq{ar}, with $U = U^k$ depending on $k$, and \eq{lah} is satisfied with
\be\label{lhe} \bar X = X^1_l.\ee

The vector $e_\la$ is again obtained from \eq{lfd}.  The required condition \eqref{ldc} now follows from
\be\label{lhf} X^1_l \ue = 0,\ee
taking
\be\label{lhg} Z^k = - X^{1\dag}_l , \; \; \, k = 1,\cdots K\ee
using \eq{lhd}, \eq{ch}.

For each $k = 1, \cdots, K$, corresponding to $Z^k$ in \eq{lhg}, there exists  $\om^k$.  Now the required condition \eq{ldd} is satisfied by $\la$ satisfying
\be\label{lhh} \suml^K_{k = 1} \la_k (\om^k\cdot \ue) = 0. \ee

As obtained, $\ue, \la$, and the matrices $U^k$ used in \eq{ar} depend on $X^1_l$.  In some examples, however, coinciding values of $\ue, \la, U^k$ are obtained for multiple values of $l$.

\subsection{Alternative coupling}

In the present framework, coupling of the given constituent systems $\Si^k(z^k), \; \; k = 1, \cdots, K$, to obtain $\Si(z)$ is obtained through the mappings
\be\label{lla} \bar z \to z \ee
as in \eq{lba}, \eq{lbc} above, and
\be\label{llb} \bar\psi(\bar z) \to \psi(z)\ee
as in \eq{lbe} above.

Sufficient conditions, \eq{ldc}, \eq{ldd} are obtained that there exists a mapping
\be\label{llc} (\bar Z, \bar \om, X) \to (Z,\om, X),\ee
as expressed above in \eq{lca}, \eq{lcb}, such that the conditions \eqref{ba}, \eq{bu} apply to $\Si(z)$.

Such preserves the gradient form \eq{bt}, the existence of an entropy extension $q(a)$,  and hyperbolicity \eq{ljc}, if present for the constituent systems.

Constituent $Z$-systems $\Si^k(z^k)$ are not required, at the expense of the potentially problematic condition \eqref{lah}.  The other qualifications \eq{laa}, \eq{lab} are made for simplicity.

An alternative coupling strategy, well-known in the application to two-fluid or multi-fluid flow, is available when the given constituent systems are $Z$-systems satisfying \eq{ca}, \eq{cc}.  A common expression for $\zeta (z)$ is assumed, and associated equations of state \eq{me}
\be\label{lma} \zeta (z^k) = \sigma_k (\xi^k(z^k), \; z^k_\zeta(z^k)), \; \; \, k = 1, \cdots, K,\ee
not necessarily coinciding but each satisfying \eq{mf}, \eq{mga}.

With this alternative strategy the condition \eq{lla} is simply
\be\label{lmg} z = \bar z;\ee
the system $\Si(z)$ is of dimension
\be\label{lmb} \bar n  = Km.\ee

Coupling is effected through a linear condition on the $\xi^k(z^k)$, \  $k = 1,\cdots, K$, equivalently on the $\zeta(z^k)$.

For some matrix
\be\label{lmc} B \in M^{K\times K},\ee
presumably chosen on physical grounds, the condition \eq{llb} now assumes the form
\begin{align}\label{lmd} \psi_i (z) &= \suml^K_{k=1} \xi^k (z^k) \; \, \big(\suml^K_{l=1} \; B_{kl}\;  \zeta_{z_i} (z^l)\big)\nonumber \\
&= \suml^K_{k=1} \;  \big(\suml^K_{l=1} \; B_{lk} \; \xi^l(z^l)\big) \zeta_{z_i} (z^k), \; \; i = 1, \cdots, m.\end{align}

The system $\Si(z)$ is obtained by application of \eq{bt} to \eq{lmd}, using \eq{lmg}, \eq{lac}. The system $\Si(z)$ is conveniently regarded as an assembly of $K$ coupled subsystems $\tilde \Si^k(z)$, each of the form \eq{aa}, \eq{ab}, \eq{ac}, \eq{bt}.  The corresponding density/flux functions are
\begin{align}\label{lme} \tilde a^k_{ij} (z) &= \psi_{i, z^k_j} (z)\nonumber \\
&= \xi^k_{z^k_j} (z^k) \; \big(\suml^K_{l = 1} \, B_{kl} \zeta_{z_i} (z^l)\big) + \big( \suml^K_{l = 1} \; B_{lk} \, \xi^l(z^l)\big) \, \zeta_{z_i z_j} (z^k)\end{align}
for $i, j = 1, \cdots, m; \; \, k = 1, \cdots, K$, with the $\xi^k(z^k), \; \xi^k_{z_j^k} (z^k)$ obtained using the given equations of state \eq{lma}.

The entropy density/flux  for $\Si(z)$ is obtained from \eq{lmd}, \eq{lac}, \eq{bs}
\begin{align}\label{lmf} \tilde q_i (\tilde a) =& \suml^K_{k = 1} \big( (\suml^m_{j = 1} z^k_j \xi^k_{z^k_j} (z^k) - \xi^k(z^k)\big) \big( \suml^K_{l = 1} \, B_{kl} \, \zeta_{z_i} (z^l)\big) \nonumber \\
&+ \big( \suml^K_{l = 1} \, B_{lk} \, \xi^l(z^l)\big) \; \big( \suml^m_{j = 1} \, z^k_j \, \zeta_{z_iz_j} (z^k))\big), \; \; i = 1,\cdots, m.\end{align}

Retaining the qualification \eq{laa}, as \eq{cu} applies to each of the $\xi^k(z^k)$, the relations \eq{llc} are now trivial
\be\label{lmh} Z = \bar Z, \; \; \om = \bar\om;\ee
the condition \eq{lah} is vacuous because of the assumed common form of $\zeta(z)$.

Hyperbolicity of each of the constituent systems $\Si^k(z^k)$ trivially implies hyperbolicity of $\bar \Si (\bar z)$, which coincides with $\Si(z)$ for the special case $B~=~I_K$.  Thus a sufficient general condition for preservation of hyperbolicity is $ \| B - I_K\|$ sufficiently small.

Alternatively, the case of $\zeta(z)$ as in \eq{eab}, with $e_H = \hat e_1$ in \eq{btc}, hyperbolicity of each $\Si^k(z^k)$ is equivalent to $\xi^k$ strictly convex in $z^k$. Then using \eq{lmd} with $i = 1$, hyperbolicity preservation results from
\be\label{lmi} B_{kl} \ge 0, \; \; \, k, l = 1, \cdots, K\ee
in \eq{lmc}.

With this coupling strategy, the resulting gradient form \eq{lmd}, existence of an entropy extension \eq{lmf} and hyperbolicity preservation are achieved at the expense of linear combinations of the $\zeta_{z_i} (z^l) $ as well as of the $\xi^l(z^l)$ in \eq{lme}.  Interpretation of the resulting system $\Si(z)$ may be affected.

As an illustrative example, we consider coupling of two isentropic Euler systems \eq{eaa}, \eq{eab}, \eq{eac} with $ m = 3$, corresponding to
\be\label{lpa} K = 2, \; \, \bar n = 6\ee
in \eq{lmb} and thereafter.

Familiar notation is recovered with abbreviations
\be\label{lpb} u_k \, \eqadef\, z^k_2, \; \, v_k\, \eqadef \, z^k_3, \; \; P_k \, \eqadef \, \xi^k, \; \; H_k \, \eqadef\, \sigma_k, \; \; k = 1, 2 \ee
and densities
\be\label{lpc} \rho_k \, \eqadef \; \frac{1}{\sigma_{k, \xi^k} (\xi^k)} \, = \, \frac{1}{H_{k, P_k} (P_k)} , \; \; \, k = 1, 2.\ee

Throughout we have tacitly assumed that the elements of the matrix $B$ preclude block partitioning.  The value of rank $B$ remains arbitrary, and determines the number of free parameters in the resulting system $\Si(z)$. Here we arbitrarily choose
\be\label{lpd} \hbox{ rank \ } B = 1, \; \; \; B = \begin{pmatrix} \alpha \beta \; \; \frac{\alpha}{\beta}\\ \frac{\beta}{\alpha} \; \; \frac{1}{\alpha \beta} \end{pmatrix} , \; \; \alpha, \beta > 0.\ee

Then from \eq{lmd}, \eq{lpd}, \eq{lpb}, for $\Si(z)$
\be\label{lpe} \psi(z) = \begin{pmatrix} (\beta + \tfrac 1 \beta) \; (\alpha P_1 + \tfrac 1 \alpha P_2) \\ (\beta u_1 + \tfrac 1 \beta u_2) \; (\alpha P_1 + \tfrac 1 \alpha P_2)\\
(\beta v_1 + \tfrac 1 \beta v_2) \; (\alpha P_1 + \tfrac 1 \alpha P_2)\end{pmatrix} .\ee

The two subsystems $\tilde \Si^k(z), \; k = 1, 2$, are of the same form, but employ velocity component averages
\be\label{lpf} \tilde u \, \eqadef \, \frac{\beta u_1 + \frac{1}{\beta} u_2}{\beta + \frac{ 1}{ \beta}} , \quad
\tilde v \, \eqadef \, \frac{\beta v_1 + \frac{ 1}{ \beta} v_2}{\beta + \frac{ 1}{ \beta}},\ee
and pressures coinciding only in the special case $\alpha = \beta =1$,
\be\label{lpg} \tilde P^1 \; \eqadef \, \frac{P_1 + \frac{1}{\alpha^2} P_2}{1 + \frac{1}{\beta^2}}
 , \; \; \;  \tilde P^2 \, \eqadef \, \frac{\alpha^2 P_1 + P_2}{\beta^2 + 1}.\ee

Each $\tilde \Si^k(z)$, trivially normalized, is then
\begin{align}\label{lph} \rho_{k, x_1} &+ (\rho_k\tilde u)_{x_2} + (\rho_k \tilde v)_{x_3} = 0 ,\nonumber \\
(\rho_ku_k)_{x_1} &+ (\rho_k u_k \tilde u + \tilde P^k)_{x_2} + (\rho_k u_k \tilde v)_{x_3} = 0 \nonumber \\
(\rho_kv_k)_{x_1} &+ (\rho_k \tilde u v_k)_{x_2} + (\rho_kv_k \tilde v + \tilde P^k)_{x_3} = 0.\end{align}

\section{Dissipation}

As reviewed in \cite{D}, entropy weak solutions $z$ of systems $\Si(z) $ \eqref{aa}, \eq{ab}, \eq{ac}, \eq{bu} are routinely sought as vanishing dissipation limits.  Approximations $z_\delta: \Om \to D, \; \, \delta > 0$, are constructed satisfying a regularized weak form of \eq{aa}, \eq{ac}
\be\label{faa} \intl_{\Om} ( \Si (z_\delta) \theta + \delta \cald (z_\de, \theta)) = 0 \ee
with test functions $\theta $ \eq{bi} and a judiciously chosen dissipation term $\cald$.  For $\cald$ such that
\be\label{fab} \de \cald (z_\de, \cdot) \overset{\de\downarrow 0}{\rhd} 0 \ee
in the sense of distributions, any sequence $\{ z_\de\} $ satisfying
\be\label{fac} z_\de \overset {\delta \downarrow 0}{\longrightarrow} z \ee
almost everywhere in $\Om$ implies
\be\label{fad} a_{ij} (z_\delta) \rhd a_{ij} (z), \; \; \, i = 1, \cdots, m, \; \, j = 1, \cdots, n\ee
weakly in $\Om$, using \eq{aa}, \eq{ab}, \eq{fab}, \eq{fac}, thus $z$ satisfying \eq{aa}, \eq{ac} weakly.

The entropy inequality \eq{bx} follows from selection of $\cald$ such that
\be\label{fae} \cald (z_\de, z_\de) \ge 0 \ee
formally.

Convergence \eq{fac} and the rate thereof, of material interest in the construction of discretization schemes, depend on suitable choice of dissipation term(s).

For systems $\Si(z)$ with a nontrivial symmetry group $\cals_R$ \eq{ad}, \eq{ae}, \eq{af}, \eq{ba}, the selection of dissipation term is (partially) directed by the observation that \eq{fac} implicitly implies
\be\label{fba} \calt_\cale (z_\de, \Om) \overset{\de\downarrow 0}{\longrightarrow} \calt_\cale (z, \Om)\ee
at least pointwise with respect to $\tau$ in \eq{af}, thus $z_\de $ satisfying
\be\label{fbb} z^*_\de = Zz_{\de} + \om \ee
in \eq{ba}.

For simplicity, attention is here restricted to the case
\be\label{fbc} c_Z = 0\ee
in \eq{bg}, \eq{bh}, \eq{bu}.

Then retaining \eq{bl} by convention,
\be\label{fbd} (\Si(z)\theta)^* = 0 \ee
for any $\cale \in \cals_R$, $z$ satisfying \eq{aa}.

Using \eq{faa}, a sufficient condition for \eq{fba} (and necessary for uniform convergence with respect to $\tau$) is a dissipation term satisfying
\be\label{fbe} (\cald (z_\de, \theta))^* = 0\ee
as determined using \eq{fbb}, \eq{bl}.

For $Z$-systems $\Si(z)$, such dissipation terms are readily constructed, independently of the related $\{ (Z, \om)\}_0$.  Using \eq{ca}, we choose
\be\label{fca} \cald (z_\de, \theta) \, \eqadef\, \big( \suml^n_{i,j = 1} \, A_{ij} (z_{\de_j, x_i} + z_{\de i, x_j} )\big) \big(\suml^n_{i, j = 1} \, A_{ij} (\theta_{j, x_i} + \theta_{i, x_j} )\big)\ee
for $A \in M^{n\times n}$, satisfying \eq{fae} by inspection.

\begin{thm} Assume the matrix $A$ symmetric and a system $\Si(z)$ satisfying \eq{ch}.

Then \eq{fbe} holds.
\end{thm}
\begin{rem*} Familiar examples include bulk  viscosity, shear viscosity and heat conduction for fluid flow models.

In view of \eq{fca}, the assumption of the matrix $A$ symmetric is without loss of generality.
\end{rem*}

\begin{proof}  Using \eq{bl}, \eq{bc}, for \eq{fca} we readily compute
\be\label{fcb} \big( \suml^n_{i,j = 1} A_{ij} (\theta_{j, x_i} + \theta_{i, x,_j} )\big)^* \, = \, \suml^n_{i, j = 1} \tilde A_{ij} \theta_{j, x_i}\ee
with $\tilde A \in M_{n\times n} $ depending on $X, Z$,
\be\label{fcc} \tilde A = Z^\dag A + AZ - AX^\dag - XA.\ee

With $A$ symmetric, $\tilde A$ is also symmetric by inspection.

However for all $X, Z$ satisfying \eq{ch}, $\tilde A$ is also antisymmetric, establishing
\be\label{fcd} \tilde A = 0\ee
independently of such $X, Z, \om$.

Using \eq{fbb}, the same argument applies to
\begin{equation*} \suml^n_{i, j = 1} A_{ij} (z_{\de j, x_i} + z_{\de i, x_j} ),
\end{equation*}
thus verifying \eq{fbe}.
\end{proof}

\section{Main theorems}

The above discussion admits alternative interpretation as part of an attempt to establish the existence of a vanishing dissipation limit \eq{faa}, \eq{fab}, \eq{fac} in the set of admissible weak solutions of systems $\Si(z)$ \eq{aa}, \eq{ab}, \eq{ac} with $m \ge 3, \, n \ge 2$.  Such interpretation derives from three tacit assumptions above.

First, that there exist sufficient structural conditions determining a (non-empty)  class of systems $\Si(z)$ within which a vanishing dissipation limit exists for judiciously selected domain and boundary data.

Second, that aside from value(s) of $n$ depending on $m$, such structural conditions may be expressed independently of $m = 2, 3, \cdots$.

Third, that such structural conditions are compatible with existence of an entropy extension \eq{bn}, \eq{bo}, \eq{bp} and a nontrivial symmetry group $\cals_R$ as determined from \eq{ba}, \eq{bg}.

Any class of systems $\Si(z)$ satisfying such conditions must admit  nontrivial solutions of \eq{bfb}, \eq{bg}, equivalently of \eq{bu} using \eq{bt}.  Such implies a linear relationship among $X,Z,\om, c_Z$ in \eq{ba}, \eq{bg}, depending on $m, n$ but independent of $z \in D$, and of $ \Si(z)$ within the specified class of systems.

The postulated class of $Z$-systems as obtained above, without restriction on the specific form of the symmetries present in $\cals_R$, is perhaps the most natural example of such a class of systems.  In particular, the condition \eq{ca} follows by attention restricted to reduced systems $\Si(z)$ as necessary, obtained by application of \eq{ai} for $m > n$ or \eq{dh} for $m < n$.

The two conditions \eq{cc}, \eq{ch} depend on \eq{ca}, and are closely related.  Compatibility thereof follows from \eq{bb} with $g(z) = \zeta(z)$, using \eq{bu}.  The condition \eq{cc} may be regarded as extension of the case $m = n = 2$, using theorem 2.1. The condition \eq{ch} then is obtained in theorem 2.2, and additionally implies \eq{bg} separately for the symmetric and antisymmetric parts of the $n$-matrix function $a(z)$ in \eq{ab}.

Possible additional systems as described in corollary 2.3 have been neglected as requiring additional restriction on the form of $\cals_R$.

Within the class of $Z$-systems, the condition \eq{so} expressing maximal symmetry may be  obtained by successive reduction of $m, n$, using  \eq{ai}, \eq{dh} as required to achieve
\be\label{taa} \La^\perp_\Si = \{ 0 \}\ee
in \eq{hlb}.  In the absence of \eq{taa}, dependence of systems $\Si(z)$ on dependent variables $e \cdot z, \; e \in \La^\perp_\Si$, remains arbitrary.

The class of $Z$-systems satisfying \eq{so} is nonempty, containing the (isentropic or extended) Euler systems.

For this class, attractive conditions on coupling and on dissipation, compatible with the symmetry group, are obtained in sections 3, 4, respectively.

In higher dimensions $m \ge 3$, hyperbolicity has been associated (albeit crudely) \cite{S2} with existence of weak solutions of attractive form for approximation. The class of hyperbolic $Z$-systems,  however, proves to be restricted to Euler systems under seemingly mild conditions.
More exotic structural conditions will apparently be required of additional system classes satisfying the above tacit assumptions.  To this extent we are corroborating the conjectured association of existence of a vanishing dissipation limit in higher dimensions with Euler and related systems.

We denote by $\Theta^{n, L}$ the set of $Z$-systems with
\be\label{tac} \rank \; a(z) = n\ee
for all $z \in D$ in \eq{ab}, with an associated equation of state \eq{me} satisfying \eq{mf}, \eq{mga}, and \eq{sp} by convention.

Elements of each $\Th^{n, L}$ additionally satisfy
\be\label{tab} L > 0\ee
in \eq{gab}, are  hyperbolic by virtue of \eq{btb}, \eq{btc}, with a time-like direction $e_H$ satisfying \eq{btd}, and for which there exists $(Z_\zeta, \om_\zeta)$ satisfying \eq{sb}, \eq{cxa}, \eq{cxb}.  Such $\Theta^{n, L} $ may be regarded as equivalence classes under transformations $\calt_{QU}$ \eq{cjz}, and such transformations are assumed applied as required throughout.

The association of classes $\Th^{n,n}$ with Euler systems is articulated in the following four theorems.

\begin{thm} Denote by $\ul{\Th}^{n,n} , \; n = 2, 3, \cdots $, the classes of systems $\Si(z)$ satisfying additionally
\be\label{tba} \dim \, \La^I_\Si = 0 \ee
in \eq{sza}, \eq{szc}, satisfying \eq{sp} and such that
\be\label{tbb} \calt_{\hat e_n, c_e} \, \Si(z) = \Si'(z') (Q'')^{-1}\ee
for almost all $c_e \in \bbr, \; \Si' (z') \in \ul{\Th}^{n-1, n-1}, \; Q'', U'' $ (of dimension $n-1$ which may depend on $c_e$) satisfying \eq{cjb}.

Then $\ul{\Th}^{n,n}$ are the isentropic Euler systems determined from \eq{eaa}, \eq{mea}.
\end{thm}
An example satisfying \eq{tba} but not \eq{tbb} is the system with $n = 4$ determined from \eqref{mea} and
\be\label{gex} \zeta(z) = z_1 + \tfrac 12 z^2_2 + (z^2_3 + z^2_4)^2.\ee

\begin{thm} Denote by ${\bar \Th}^{n, n}, \; n = 3,4,\cdots$, the classes of systems $\Si(z)$ satisfying
\be\label{tbc} \dim \, \La^I_\Si = 1, \; \; \La^I_\Si = span \{ \hat e_n\},\ee
with associated equation of state \eq{me}, with
\be\label{tbd} z_\zeta (z) = z_n, \; \; \, e_{z_\zeta} = \hat e_n\ee
in \eq{bwa}, and satisfying \eq{sb} and  \eq{tbb} with $c_e$ satisfying \eq{dha} and
$\Si'(z') $ an isentropic Euler system of dimension $n-1$.

Then $\bar{\Th}^{n, n}$ are the extended Euler systems \eq{ecc}, \eq{ecd}, \eq{ece}, possibly with additional additive terms $(a_{nj} (z))_{x_n}, \; \, j = 1, \cdots, n$.
\end{thm}
\begin{thm} Each class of systems
$\Si(z) \in \Th^{n,n}, \; n = 4, 5, \cdots$, satisfying
\be\label{tbe}  \dim \La^I_\Si = 2, \ee
and satisfying \eq{tbb} for almost all $c_e \in \bbr$ with $\Si'(z')$ an extended Euler system of dimension $n-1$, is empty.
\end{thm}
This does not preclude $\Th^{n,L}, \, L < n$ satisfying \eq{tbe}.
\begin{thm} Denote by $\overset * \Th{}^{n, n-1}$ the classes of systems $\overset * \Si (z) $ with
\be\label{tbf} \dim \La^\perp_\Si = 1, \; \; \La^\perp_\Si = span \{ \hat e_n\},\ee
satisfying \eq{tbd} and \eq{tbb} with $\Si' (z')$ an isentropic Euler system of dimension $n-1$.

Then $\overset * {{\Th}}{}^{n, n-1} $ are the systems \eq{edd}, \eq{ede}, conserving entropy as opposed to energy, again possibly with additional additive terms $(a_{n_j} (z))_{x_n}$,\linebreak  $j = 1, \cdots, n$.
\end{thm}


\begin{rema*} In each case further details on the equation of state remain unspecified.
\end{rema*}

The proofs of these four theorems are closely related and are deferred to the following section.

\begin{cor} Assume
\be\label{tbi}  \Si(z) \in\Theta^{n, L} , \; \; n \ge 3, \; \; 2 \le L < n\ee
such that
\be\label{tbj} \calt_{\hat e_{L + 1}, c_{L+1}} \cdots \calt_{\hat e_n, c_n} \Si (z) = \Si' (z') \in \Theta^{L,L}\ee
with
\be\label{tbk} \hat e_{L+1}, \cdots, \hat e_n \in \La^\perp_\Si\ee
and some $c_{L+1}, \cdots, c_n \in \bbr$.

Then $\Si' (z') $  is an Euler system,
\be\label{tbl} \Si'(z') \in \ul{\Theta}^{L,L} \cup \bar\Theta^{L,L}.\ee
\end{cor}

\section{Proofs of the main theorems}

The proofs are obtained by induction on $n$.  Elementary phase plane analysis suffices to prove theorem 5.1 in the special case $n = 2$.  Subsequently introduced subclasses of $\Theta^{n,L}$ permit construction of right inverse mappings of $\calt_{e, c_e}$.

\begin{lem} Assume
\be\label{gaa} \Si(z) \in \Theta^{n, L}, \; \; \, n \ge 2 \ee
satisfying \eq{tac} and \eq{tab} with $L $ obtained from \eq{gab}.  Then
\be\label{gac} L \ge 2.\ee
\end{lem}
\begin{proof} Assume \eq{gac} fails, implying $L = 1$ using \eq{tab}.  Applying a transformation $\calt_{QU}$ as required, there exists $k = 1, \cdots, n$ such that \eq{hh} holds with
\be\label{gae} \La_{\Si(z)} = span \{ \hat e_k\}.\ee

Then from \eq{tab}, there exists
\be\label{gbt} (Z,\om) \in \{ (Z, \om)\}_0.\ee

For all such $(Z,\om), \; \,y \in D$, using \eq{cs}, \eq{cu}, \eq{gae},  for the system $\Si(z)$,
\begin{align}\label{gad} \zeta_z (y) (Z y + \om) &= \zeta_{(\hat e_k \cdot z)} (y)\;  ((Zy)_k + \om_k)\nonumber \\
&= 0.\end{align}

With $L = 1$, using \eq{gae}, there can be no $e \in \bbr^n$ satisfying \eq{szc}, \eq{szd}, so necessarily \eq{gae} holds with
\be\label{gaz} \hat e_k \in\La^V_{\Si(z)}\ee
satisfying \eq{szb}.

From \eq{gad}, \eq{gaz}, \eq{szb},
\be\label{gam} \zeta_{(\hat e_k\cdot z)} (y) = 0 \ee
for almost all $y\in D$, and \eq{cga} holds with
\be\label{gai} e_\perp = \hat e_k.\ee

If $\xi(z)$ is specified independently of $\zeta(z)$ for the system $\Si(z),$ then using \eq{cr}, \eq{ci} the conditions \eq{gad}, \eq{gam} also hold with $\zeta$ replaced by $\xi$, and \eq{tac} fails using \eq{cc}, \eq{bt}.

For $\xi (z)$ obtained from $\zeta(z)$ and an equation of state \eq{me}, the condition \eq{tac} can hold only if
\be\label{gan} z_{\zeta, (\hat e_k \cdot z)} (y) \neq 0 \ee
for almost all $y\in D$.

But \eq{gan}, \eq{gaz}, \eq{mga} are incompatible.
\end{proof}

Lemma 6.1 does not depend on hyperbolicity. But making transformations $\calt_{QU}$ as required, throughout we retain \eq{sp} by convention in \eq{btb}, \eq{btc}, \eq{btd}.

\begin{lem}  The systems $\Theta^{2,L}$ are isentropic Euler systems, satisfying \eq{eaa}, \eq{mea}.
\end{lem}
\begin{proof} By appeal to lemma 6.1, it suffices to consider systems
\be\label{gda} \Si(z) \in \Th^{2,2}.\ee

For such systems, from \eq{cj}, \eq{cu}, the prescribed function $\zeta(z)$ satisfies
\be\label{gdb} (\zeta_{z_1} (y)\;\; \; \; \zeta_{z_2} (y)) (Zy+\om) = 0\ee
for all $y\in D, (Z,\om)$ satisfying \eq{gbt}, equivalently
\be\label{gdc} \frac{d}{d\tau} \, \zeta(y,(\tau)) = 0 \ee
on trajectories within $D$
\be\label{gdd} y_\tau (y(\tau)) = Z y(\tau) + \om.\ee

Thus $(Z,\om)$ is unique up to scaling.

From \eq{btd}, $Z$ is singular.  Vanishing $Z$ would imply nontrivial
\be\label{gde} e_\perp \in \La^\perp_{\Si(z)},  \; \, e_\perp \cdot \om = 0,\ee
incompatible with \eq{gda}.  Thus necessarily
\be\label{gdg} \rank\;  Z = 1.\ee

Using \eq{btd}, \eq{sp}, it suffices to consider
\be\label{gdh} Z = \begin{pmatrix} 0 & 0\\0 &1\end{pmatrix},\ee
or alternatively
\be\label{gdi} Z = \begin{pmatrix} 0 & 1\\0 &0\end{pmatrix}.\ee

For \eq{gdh}, solution of \eq{gdb} is
\be\label{gdj} \zeta(z) = z_2 e^{-z_1} , \; \om = {1\choose 0}.\ee

For \eq{gdi}, from \eq{gdb}
\be\label{gdk} \zeta(z) = z_1 + \thalf \; z^2_2, \; \om = {0\choose -1}.\ee

In either case, the conditions \eq{cxa}, \eq{cxb} are satisfied with
\be\label{gdl} Z_\zeta = 0,\; \om_\zeta = {1\choose 0}.\ee

By appeal to theorem 2.10, such systems $\Si(z)$ are associated with an equation of state \eq{me}.  But as
\be\label{gdm} \dim \La^I_\Si = \dim\La^\perp_\Si = 0 \ee
follows from \eq{gdh}, \eq{gdk}, necessarily $z_\zeta (z)$ vanishes identically, using \eq{sc}, and the equation of state is of the form \eq{mea}. In the case \eq{gdj}, hyperbolicity fails.  In contrast, in the case \eq{gdk}, using \eq{sp}, the condition \eq{bwc} suffices for hyperbolicity.

Such $\Si(z)$ are precisely the isentropic Euler systems \eq{eaa} with $n~=~2$.
\end{proof}

\subsection{System subclasses}

Subclasses of systems $\Th^{n,L}$, satisfying additional structural or symmetry conditions, are distinguished by subscripts below; intersections thereof by multiple subscripts.  Each such subclass is invariant under transformations $\calt_{QU}$ \eq{dja}, satisfying \eq{cjb} and maintaining \eq{sp}, satisfying
\be\label{gap} Q_{1j} = U_{1j} =  \de_{1, j}, \; \; j = 1,\cdots, n.\ee

Additionally, these subclasses satisfy useful properties with respect to transformations $\calt_{e, c_e}$ \eq{cjd}, maintaining \eq{sp} with
\be\label{gar} e = \hat e_j, \; \; j \ge 2, \; \,   n \ge 3,\ee
and transformations $\calt_{q, l , c_e}$ \eq{dja} with
\be\label{gas} l, n \ge 3.\ee

For systems
\be\label{gba} \Si(z) \in \Th^{n,L}_W,\ee
the prescribed function $\zeta(z)$  in \eq{cc} is of the form \eq{hdy}, using \eq{hba},  with neither $W$ nor $Y$ vanishing identically.

Applying a transformation $\calt_{QU}$ and scaling $z$ as necessary, the condition \eq{sp} is maintained with $W$ diagonal below,
\be\label{gak} W_{1} = 0, \, Y_1 = 1; \; W_{jj} \in \{ 0, 1\}, \; Y_j = 0, \; j = 2, \cdots, n.\ee

In contrast, for
\be\label{gbb} \Si(z) \in \Th^{n,L}_C,\ee
the system is closed as per definition 2.8, as in \eq{ebd}.

For systems
\be\label{gbc} \Si(z) \in \Th^{n,L}_\perp\ee
there exists a nontrivial $e_\perp$ in \eq{cga}.  Such systems cannot be uniformly hyperbolic as per definition 2.12.

For systems
\be\label{gbe} \Si(z) \in \Th^{n,L}_T,\ee
the equation of state is of the form \eq{me}, \eq{mga}.

For
\be\label{gbd}  \Si(z) \in \Th^{n,L}_{T^*},\ee
the equation of state is of the form \eq{mea}.

For systems
\be\label{gbf} \Si(z) \in \Th^{n, L}_q,\ee
the conditions \eq{me}, \eq{bwa} hold and the entropy flux and potential vectors are related by
\be\label{gbg} q = - \frac{\sigma_{z_\zeta}}{\xi \sigma_\xi}\;  \psi\ee
as in \eq{ecf}

In contrast, for systems
\be\label{gbh} \Si(z) \in \Th^{n,L}_{q^*},\ee
the entropy flux vector vanishes identically.

For systems
\be\label{gbi}\Si (z) \in \Th^{n, L}_I,\ee
the subspace $\La^I_\Si$ in \eq{sza}, satisfying \eq{szc}, \eq{szd}, is nontrivial, as opposed to
\be\label{gbl} \Si(z) \in \Th^{n,L}_{I^*},\ee
for which $\La^I_\Si$ is trivial; in \eq{sza},
\be\label{gbla} \La_\Si = \La^V_\Si.\ee

For systems
\be\label{gbo} \Si(z)\in \Th^{n,L}_\om\ee
there exists $(Z,\om)$ satisfying \eq{gbt} with
\be\label{gboa} \om \neq 0,\ee
not vanishing identically, whereas for
\be\label{gbp} \Si(z) \in \Th^{n,L}_{\om^*},\ee
$\om$ vanishes identically; \eq{gbt} holds with
\be\label{gbq} \{ ( Z, \om)\}_0 = \{ (Z, 0)\}_0.\ee

In the special case that \eq{so} holds, $\La^\perp_\Si$ in \eq{gab} is trivial, and $\La_\Si = \bbr^n$ in \eq{sza}.

In this case, from \eq{gbe}, \eq{mga}, \eq{gbi}
\be\label{gca} \Th^{n,n}_T \subset \Th^{n,n}_I\ee
and conversely
\be\label{gcb} \Th^{n,n}_{I^*} \subset \Th^{n,n}_{T^*}.\ee

From \eq{gbb}, \eq{szc}, \eq{szd}, \eq{gbi}
\be\label{gcc} \Th^{n, n}_C \subset \Th^{n,n}_I.\ee

From \eq{gbg}, \eq{gbe},
\be\label{gcz}  \Th^{n,L}_q \subset \Th^{n,L}_T.\ee

\begin{lem} For $\Th^{n,L}_I$,
\be\label{gga} e_H \notin \La^I_\Si.\ee
\end{lem}
\begin{rema*} For $L < n, \, e_H \in \La^\perp_\Si$ is possible, for example
\[ \Si (z) \in \Th^{3,2}_{I^*}, \; \zeta(z) = z_1 + (z^2_2 + z^2_3)^2.\]
\end{rema*}
\begin{proof} The condition \eq{btd} for $e_H$ is incompatible with \eq{szc}, \eq{szd}.
\end{proof}

\begin{lem} In \eq{gba}
\be\label{gxb} \Th^{n,L}_W \subset \Th^{n,L}_{I^*, \om}\ee
and
\be\label{gxa} L = 1 + \rank \, W.\ee
\end{lem}
\begin{rema*} For $n =2$, this follows by inspection using \eq{gdi}, \eq{gdk}.
\end{rema*}
\begin{proof} For $\Si(z)$ satisfying \eq{gba}, \eq{gak}, denote
\be\label{gxg} J_\Si\, \eqadef\, \{ j = 2, \cdots, n \, | \, W_{jj} = 1 \}.\ee

From  \eq{he}, \eq{gxg}, \eq{hd}, the nonzero elements of $Z$ satisfying \eq{gbt} are limited to
\be\label{gxd} Z_{1j}, \; j \in J_\Si\ee
and
\be\label{gxf} Z_{jl}, \, Z_{lj}, \; j \neq l, \; j, l \in J_\Si.\ee

From \eq{hea}, \eq{heb}, \eq{hec}, \eq{gxd}, \eq{gxg}, there exists $(Z, \om_Z)$ satisfying \eq{gbt} such that
\be\label{gxe} \om_{Z, j} \neq 0, \; \; j \in J_\Si,\ee
thus establishing \eq{gboa}.

From \eq{gxf}, \eq{gxd}, \eq{gxg}, \eq{szb}
\be\label{gxc} \hat e_j \in \La^V_\Si, \; \, j = 1 \; \hbox{or \ } j \in J_\Si.\ee

Thus $\La^I_\Si $ obtained from \eq{szc} is trivial.
\be\label{gxh} \La^\perp_\Si = span \{ \hat e_j, \, j \neq 1, \; j \notin J_\Si\}\ee
and \eq{gxa} follows from \eq{gab}, \eq{gxc}, \eq{gxh}.
\end{proof}
\begin{lem} No generality is lost from assuming
\be\label{gya} \Th^{n,n}_{I^*, \, \perp} = \{ 0 \}.\ee
\end{lem}
\begin{proof} Assume nontrivial $e_\perp$ in \eq{cga}.  For $\Si(z) \in \Th^{n,n}_{I^*}$, from \eq{gbl}, \eq{gab}, necessarily
\be\label{gyb} e_\perp \in \La^V_\Si.\ee

But from \eq{cga}, \eq{gcb}, for all $z \in D, \; \, \zeta(z)^*$ and $\xi(z)^*$ are independent of $e_\perp \cdot z$, so no loss of generality results from assuming
\be\label{gyc} Z^\dag e_\perp = \om \cdot e_\perp = 0 \ee
for all $(Z,\om) \in \{ ( Z,\om)\}_0$, precluding \eqref{gyb}.
\end{proof}

\subsection{The isentropic systems}

In this notation $\Th^{n, n}_{I^*, W, \om}$ designates the isentropic Euler system \eq{eaa}, \eq{eab}, \eq{eac} of dimension $n\ge 2$.  Using \eq{hdy}, \eq{gak} these systems are understood with $\zeta(z)$ determined from \eq{eab}; the equation of state is \eq{mea}.

The extended Euler systems \eq{ebb}, \eq{ebc}, \eq{ebd},  \eq{ecc}, \eq{ecd}, \eq{ece} of dimension $n \ge 3$ are designated $\Th^{n,n}_{C, T, \om^*, q}$.

The systems \eq{edc}, \eq{edd}, \eq{ede} of dimension $n \ge 3$ conserving entropy as opposed to energy, are designated $\Th^{n, n-1}_{W, \perp, T, I^*, \om}$.

Using \eq{gbi}, \eq{gbl} and then \eq{gca}, for $\Si(z) \in \Th^{n,n}$ in \eq{tbb} we use
\begin{align}\label{gff} \Th^{n,n} &= \Th^{n,n}_I \cup \Th^{n,n}_{I^*}\nonumber \\
&=\Th^{n,n}_T \cup \Th^{n,n}_{I, T^*}\cup \Th^{n,n}_{I^*}.\end{align}

From lemmas 6.2, 6.4, in the statement of theorem 5.1,
\be\label{gax} \ul{\Th}^{2,2} = \Th^{2,2}_W = \Th^{2,2}_{I^*, W, \om}.\ee

Using \eq{gax}, theorem 5.1 is proved by induction on $n$. Using \eq{gff}, it suffices to prove
\begin{lem} Assume $\Si (z), \Si'(z')$ satisfying \eq{tbb} with
\be\label{gha} \Si(z) \in \Th^{n,n}_{I^*}, \; \; \Si' (z') \in \Th^{n-1, n-1}_{I^*, W, \om}.\ee

Then
\be\label{ghb} \Si(z) \in \Th^{n, n}_{I^*, W, \om}.\ee
\end{lem}

\begin{proof} From \eq{tbb}, comparing \eq{tbb}, \eq{uaa} and using \eq{aq}, \eq{anb}, it follows that the dependent variables $z, z'$ are related by
\be\label{ghc} \begin{pmatrix} z_1\\ \vdots\\ z_{n-1}\end{pmatrix} \, = \, Q'' (c_e) z'\ee
with some nonsingular $Q''(c_e)$. As $c_e$ is arbitrary, by convention \eq{cje} the condition  \eq{ghc} can hold only with $z, z'$ related by
\be\label{ghd} \begin{pmatrix} z_1\\ \vdots\\ z_{n-1}\end{pmatrix} \, = \, Q''(z_n) z'.\ee

From the invariance of $\Si(z), \; \Si'(z')$ under transformations $\{ (Z,\om)\}_0, \{ (Z',\om')\}_0$, necessarily for any such transformations
\be\label{ghe} (Q''(z_n))^* = 0.\ee

From \eq{gha}, $z^*_n$ cannot vanish identically, so \eq{ghe} can hold only with $Q''$ independent of $z_n$.  Without loss of generality, we take
\be\label{ghf} Q''= I_{n-1}.\ee

The systems $\Si(z), \Si'(z')$ are related by \eq{uab}, with \eq{hdy}, \eq{gak} applicable to $\Si'(z')$.  Using \eq{ghc}, \eq{ghf}, necessarily
\begin{align}\label{ghj} \frac{\partial^2 \zeta (z', c_e)}{\partial z'_i \, \partial z'_j} &= 0 , \; \; \, i, j = 1, \cdots, n-1, \; \; i \neq j;\nonumber \\
\frac{\partial^2\zeta(z', c_e)}{(\partial z'_1)^2} &= 0;\nonumber \\
\frac{\partial^3 \zeta(z', c_e)}{(\partial z'_j)^3} &= 0, \;\; j = 2, \cdots, n-1.\end{align}

As \eq{ghj} must hold for almost all $c_e \in \bbr, \; z' \in D'$, these expressions vanish identically, and $\zeta(z)$ is necessarily of the form
\be\label{ghk} \zeta(z) = \phi_1 (z_n) z_1 + \tfrac 12 \suml^{n-1}_{j = 2} \phi_j (z_n) z^2_j + \phi_n (z_n)\ee
with scalar functions $\phi_j(z_n), \; \, j = 1, \cdots, n$ of class $C^2$.

The symmetry groups $\{ ( Z, \om)\}_0$ for $\Si(z)$ and $\Si'(z')$ are related by \eq{cjf}, \eq{cjg}, \eq{cjh} independently of $z, z'_j\; z_n$ in particular.  As $z^*_n$ cannot vanish identically,  \eq{ghk} can hold only with
\begin{align}\label{ghg} \zeta(z) &= \zeta'(z') + \phi_n (z_n)\nonumber \\
&=z_1 + \tfrac 12 \suml^{n-1}_{j = 2} z^2_j + \phi_n (z_n)\end{align}
again using \eq{hdy}, \eq{gak}, \eq{ghc}, \eq{ghf}.

From \eq{gha}, \eq{sza}, necessarily for $\Si(z)$
\be\label{ghl} \hat e_n \in \La^V_{\Si(z)}\ee
and \eq{ghg} can hold only with
\be\label{ghh} \phi_n(z_n) = c z^2_n,\; \; c \neq 0, \ee
implying
\be\label{ghi} \Si(z) \in \Th^{2,2}_W.\ee

Now \eq{ghb} follows from \eq{ghi} using lemma 6.4.
\end{proof}

\subsection{Systems with an extended temperature variable}
The proofs of theorems 5.2, 5.3, 5.4 are combined, as they relate to $Z$-systems
\be\label{gja} \Si(z) \in \Th^{n,L,\calt}_T, \; \; \;   n \ge 3, \; \, L = 2, \cdots, n\ee
with an associated equation of state of the form \eq{me}, satisfying \eq{mf}, \eq{mga}, \eq{tbb}, \eq{tbd}.

From \eq{mga}, \eq{tbd}, \eq{szb}, necessarily in \eq{gja}
\be\label{gjb} \hat e_n \notin \La^V_{\Si(z)}.\ee

Either \eq{gja} holds with $L = n$, satisfying \eq{tbc} and
\be\label{gjc} \Si(z) \in \Th^{n, n,\calt}_{I, T} =\bar\Th^{n,n}\ee
as introduced in theorem 5.2, or else $L = n-1$, \eq{tbf} holds and
\be\label{gjd} \Si(z) \in \Th^{n, n-1, \calt}_{\perp, I^*, T} = \Th^{*n, n-1} \ee
as introduced in theorem 5.4.  In either case, using \eq{mga}, \eq{tbd}, \eq{ebb}, \eq{edd}, $z_n$ is identified as an extended temperature variable.

At issue is characterization of the expression for $\zeta(z)$ in each case, exploiting shared symmetry group properties.  Satisfying \eq{tbb}, \eq{tbd}, each system class \eq{gjc}, \eq{gjd}  inherits a symmetry group $\{ (Z, \om)\}_0$ from that of the image $\Si'(z')$, as determined from \eq{cjf}, \eq{cjg}, \eq{cjh}, \eq{ar}.  The symmetry group $\{ ( Z, \om)\}_0$ for $\Si'(z')$ is obtained from \eq{hd}, using \eq{he}, \eq{hez}, \eq{hea}, \eq{heb}, \eq{hec}.  For simplicity of notation, without loss of generality we use \eq{gak}.

For each system class, the argument \eq{ghc}, \eq{ghd} \eq{ghe}, \eq{ghf} applies, so $\zeta(z)$ is of the form \eq{ghk}.  In each case the functions $\phi_1 (z_n), \cdots, \phi_n(z_n)$ are determined by details on the symmetry group.

The symmetry group for $\Si'(z')$ contains ``rotation" symmetries, with $Z$ antisymmetric,
\be\label{gkz} Z_{1j}, Z_{j1}, Z_{nj}, Z_{jn}, \om_j = 0, \; \; \, j = 1, \cdots, n,\ee
implying that in each case, \eq{ghk} is compatible with \eq{cj}, \eq{cu} only with
\be\label{gka} \phi_2 (z_n) = \cdots = \phi_{n-1} (z_n).\ee

The symmetry group for $\Si'(z')$ also includes ``Galilean" symmetries, with the only nonvanishing components $Z_{1j}, Z_{jn}, \om_j,  \,  j = 1, \cdots, n$, satisfying \eq{hea}, \eq{heb}, \eq{hec}.

Using \eq{ba}, \eq{cjf}, \eq{gka}, \eq{ghk} in \eq{cj}, \eq{cu}, for each system class the functions $\phi_1(z_n), \phi_2 (z_n)$ are restricted by
\be\label{gkb} \phi_1 (y_n) \Big( \suml^n_{j = 2} Z_{1j} y_j + \om_1\Big) + \phi_2 (y_n) \Big(\suml^{n-1}_{j = 2} y_j (Z_{jn} y_n + \om_j)\Big) = 0 \ee
for all $y \in D, \, Z_{1j}, Z_{jn}, \om_j$ satisfying \eq{hea}, \eq{heb}, \eq{hec}.

By inspection, the condition \eq{gkb} can hold only with
\be\label{gkc} Z_{1n} = \om_1 = 0\ee
and either
\be\label{gkd} \phi_1 (y_n) = \phi_2 (y_n); Z_{jn} = 0,\;  \om_j = - Z_{1j}, \; \; j = 2, \cdots, n-1;\ee
or else
\be\label{gke} \phi_1 (y_n) = 2y_n \phi_2 (y_n); Z_{jn} = - Z_{1j}, \om_j = 0,  \; j = 2, \cdots, n-1.\ee

Using \eq{szd}, \eq{gboa}, the alternative \eq{gkd} implies \eq{gbl}, \eq{gbo}, while \eq{gke} implies \eq{gbi}, \eq{gbp}.  Thus \eq{gjc} is refined to
\be\label{gkf} \Si(z) \in \Th^{n,n,\calt}_{I, T, \om^*} \ee
with \eq{gke} holding, and \eq{gjd} becomes
\be\label{gkg} \Si(z) \in \Th^{n, n-1, \calt}_{\perp, I^*, T, \om}\ee
with \eq{gkd} holding.

For systems \eq{gkg}, from \eq{ghk}, \eq{gkd},
\be\label{gkj} \zeta(z) = \phi_2(z_n) z_1 + \phi_2 (z_n) \suml^{n-1}_{j = 2}  z^2_j + \phi_n (z_n).\ee

For systems \eq{gkf}, from \eq{ghk}, \eq{gke}, we distinguish
\be\label{gkl} \hat \zeta (z) = 2 z_n \hat \phi_2 (z_n) z_1 + \hat \phi_2 (z_n) \suml^{n-1}_{j = 2} z^2_j + \phi_n (z_n).\ee

In both cases, \eq{sb}, \eq{cxa}, \eq{cxb} are satisfied with either
\[ Z_\zeta = 0, \; \; \om_\zeta = \begin{pmatrix}1\\0\\ \vdots\\0 \end{pmatrix}\]
or with
\[ Z_{\zeta, 1n} = 1, \; \; \, \om_\zeta = 0,\]
the other components of $Z_\zeta$ vanishing.  Thus theorem 2.10 applies in both cases, and an associated equation of state \eq{me} is presumed.

The two system classes \eq{gkf}, \eq{gkg} are necessarily related.
With $c_e$ satisfying \eq{dha}, setting $ l = n$ in \eq{dja}, \eq{dka}, using \eq{dkb}, \eq{dkc}, \eq{cjc},
\be\label{gla}\calt_{\hat e_n, c_e} \calt_{q, n, c_e} = \calt_{\hat e_n, c_e}.\ee

The alternative \eq{gbe}, \eq{gbd} is invariant under $\calt_{q, n, c_e}, $ so from \eq{gla}, \eq{gkf}, \eq{gkg}
\be\label{glb} \calt_{q, n, c_e} : \big(\Th^{n,n,\calt}_{I,T,\om^*} \cup \Th^{n,n-1,\calt}_{\perp, I^*,T,\om}
\big) \to \big(\Th^{n,n,\calt}_{I,T,\om^*} \cup \Th^{n,n-1,\calt}_{\perp, I^*,T,\om}\big).
\ee


In view of \eq{gbi}, \eq{gbl}, the condition \eq{glb} implies
\be\label{glc} \Th^{n, n, \calt}_{I,T,\om^*} \; = \;\calt_{q, n, c_e} \; \Th^{n, n-1, \calt}_{\perp, I^*, T, \om},\ee
of the form \eq{dja}.  Thus \eq{gkj}, \eq{gkl} must be compatible with \eq{dkb}, \eq{dkc}, \eq{dkd}, \eq{cjc}.

Using \eq{ube}, without loss of generality we identify systems \eq{gkg}, \eq{gkj} with $\Si(z)$  and systems \eq{gkf}, \eq{gkl} with $\tilde \Si (\tilde z)$ in \eq{dkb}, \eq{dkc}, \eq{cjc}.  Using \eq{dha}, we take
\be\label{gky} z_n \neq 0,\;  \sgn(z_n) = \sgn(c_e), \; \; z \in D.\ee

Then \eq{gkj}, \eq{gkl} are necessarily related by
\be\label{glf} \hat\phi_2 (z_n)
=
\frac{\phi_2 \big(\frac{c^2_e}{z_n}\big)}{z^2_n},\;\;\, \hat \phi_n (z_n) = \phi_n \bigg(\frac{c^2_e}{z_n}\bigg).\ee

From \eq{glc}, hyperbolicity is equivalent for the classes \eq{gkf}, \eq{gkg}.  For the class \eq{gkg}, by inspection necessary conditions for hyperbolicity include \eq{bwc}, \eq{bwd}.  Using \eq{gky}, the condition \eq{bwd} fails unless $\phi_2(z_n)$ is independent of $z_n$; for simplicity we take
\be\label{glg} \phi_n (z_n) = 1.\ee

Then using \eqref{glg} in \eq{gkj},
\be\label{gld} \zeta (z) = z_1 + \tfrac 12 \suml^{n-1}_{j = 2} \, z^2_j + \phi_n(z_n)\ee
for the systems \eq{gkg}.  Using \eq{glf}, \eq{glg} in \eq{gkl},
\be\label{gle} \hat \zeta(z) = \frac{z_1}{z_n} + \half \suml^{n-1}_{j = 2} \bigg(\frac{z_1}{z_n}\bigg)^2 + \phi_n \bigg(\frac{c^2_e}{z_n}\bigg)\ee
for the systems \eq{gkf}.

The function $\phi_n (\cdot)$ remains undetermined in \eq{gld}, \eq{gle}.

\begin{lem} For either class of systems in \eq{glc}, sufficient conditions for hyperbolicity are \eq{bwc}, \eq{bwf} for the equation of state \eq{me}, and $\phi_n$ satisfying
\be\label{gma} \phi_{n, z_n z_n} (z_n) \ge 0, \; \; \, z \in D.\ee
\end{lem}
\begin{proof} As the transformation $\calt_{q, n, c_e}$ preserves hyperbolicity, using \eq{glc} it suffice to consider $\zeta(z) $ obtained from \eq{gld}. In this case, using \eq{sp}, the conditions \eq{bwg}, \eq{bwe} are immediate, and \eqref{bwd} follows from \eq{gld}, \eq{gma}.
\end{proof}

Extended and related Euler systems are obtained in the special case
\be\label{gmb} \phi_n (z_n) = 0.\ee


For the systems \eq{gld}, the conditions \eq{gba}, \eq{gbc} follow from \eq{gmb} by inspection.  Using  \eq{cc}, \eq{bt}, \eq{gld}, \eq{gmb}, for these systems
\be\label{gmc} \psi_n (z) = 0, \; \; a_{nj} (z) = 0, \; \; j = 1, \cdots, n,\ee
and using \eq{gld}, \eq{me} we readily obtain
\begin{align}\label{gmd} a_{in} (z) &= - \frac{\sigma_{z_\zeta} (\xi(z), z_\zeta(z))}{\sigma_\xi (\xi (z), \, z_\zeta (z))} \frac{\psi_i(z)}{\xi(z)}\nonumber \\
&= a_{i1} (z) \sigma_{z_\zeta} (\xi(z), z_\zeta (z)), \; \; \, i = 1, \cdots, n-1\end{align}

These are the systems \eq{edd}, \eq{ede}, conserving entropy  as opposed to energy.

For the systems \eq{gle}, the condition \eq{gbb} follows from \eq{gle}, \eq{gmb}, and \eq{gbf}, \eq{gbg} from \eq{glc}, \eq{djz}, \eq{gmd}.

These are the extended Euler systems \eq{ecc}, \eq{ecd}, \eq{ece} (where $z_n$ has been replaced by $-z_n$ in \eq{ebb} for simplicity in identification).

For either system class, hyperbolicity survives relaxation of the condition \eq{gmb} to \eq{gma}.
 Such results only in additional additive terms in $\psi_n (z)$, from \eq{cc}, and additional additive terms in $a_{nj } (z), \; \, j = 1, \cdots, n$, from \eq{bt}. Indeed, the systems so obtained are very similar to those obtained omitting $\phi_n$ from \eq{gld} or \eq{gle} with compensation in the corresponding equation of state, replacing $\sigma (\xi, z_n)$ by $\sigma (\xi, z_n) - \phi_n (z_n)$ or $\sigma(\xi, z_n) - \phi_n \Big(\frac{c}{z_n}\Big)$ respectively.

The proofs of theorem 5.2 and theorem 5.4 are thus completed.

The conclusion of theorem 5.3 is now immediate. Were \eq{tbe} to hold, for $e\in \La^I_\Si$, using \eq{cjh} we would have $\Si(z)$ satisfying
\be\label{gme} \calt_{e, c_e} \Si(z) \in \Th^{n-1, n-1, \calt}_{I, T, \om}\ee
contradicting \eq{gkf} with $n-1$ replacing $n$.



\newpage


\begin{thebibliography}{123}


\bibitem[BK]{BK} G.W. Bluman and S. Kumei, \emph{Symmetries and Differential Equations}, Springer-Verlag, New York (1989).

\bibitem[D]{D} C.M. Dafermos, \emph{Hyperbolic conservation laws in continuum physics}, Springer-Verlag, Berlin (2000).


\bibitem[FL]{FL} K.O. Friedrichs and P.D. Lax, \emph{Systems of conservation laws with a convex extension}, Proc. Nat. Acad. Sci {\bf 68} (1971), pp. 1686--1688.


\bibitem[G]{G} S.K. Godunov, \emph{An interesting class of quasilinear systems}, Dokl. Akad. Nauk SSSR {\bf 139} (1961), pp. 521--523.


\bibitem[L]{L} P.D. Lax, \emph{Shock waves and entropy}, in \emph{Contributions to Functional Analysis}, E.A. Zarontonello, ed., Academic Press, 1971.


\bibitem[M]{M} M.S. Mock, \emph{Systems of conservation laws of mixed type}, J. Diff. Eq. {\bf 37} (1980), pp. 70--88.


\bibitem[S1]{S1} M. Sever, \emph{Canonical form and symmetry group of systems of conservation laws}, J. Diff. Eq. {\bf 255} (2013), pp. 1607--1645.


\bibitem[S2]
{S2} M. Sever, \emph{Systems of conservation laws in higher space dimensions}, arXiv:2408-14896 (2004).







\end{thebibliography}
\end{document}